\documentclass{amsart} 

\usepackage{multirow}
\usepackage{amssymb}
\usepackage{mathabx}
\usepackage{eucal} 
\usepackage{tikz}
\usetikzlibrary{positioning,fit}
\usepackage[bookmarks,colorlinks=true]{hyperref}

\newcommand \seq[1]{{\left\langle{#1}\right\rangle}}
 
\newcommand{\rest}[1]{\!\!\upharpoonright\!{#1}} 
\newcommand \tth{{}^{\textup{th}}}
\DeclareMathOperator \dom{dom}
\newcommand{\cero}{\mathbf{0}}
\newcommand{\converge}{\!\!\downarrow}
\newcommand{\diverge}{\!\!\uparrow}
\newcommand{\w}{\omega}
\newcommand \s{\sigma}
\renewcommand \le {\leqslant}
\renewcommand \ge {\geqslant}
\newcommand \vphi{\varphi}
\newcommand \Tur{\textup{T}}
\newcommand \SJT{\textup{SJT}}
\newcommand \finn{\textup{\texttt{fin}}}

\renewcommand \top{\textup{\texttt{top}}}

\newcommand \PPP{\mathcal P}
\newcommand \NNN{\mathcal N}
\let\littlehat\hat
\renewcommand \hat{\widehat}

\renewcommand \bar{\overline}
\newcommand \conc[2]{{#1}\littlehat{\,\,}{#2}}
\newcommand \Par{\mathtt{Par}}

\newcommand \stage{\emph{stage} }

\newcommand\ba{{\mathbf{a}}}

\newcommand \LR{\textup{LR}}

\theoremstyle{plain}
\newtheorem{theorem}{Theorem}[section]

\newtheorem{lemma}[theorem]{Lemma} 
\newtheorem{corollary}[theorem]{Corollary}

\theoremstyle{definition}
\newtheorem{definition}[theorem]{Definition} 

\theoremstyle{remark}

\title{Pseudo-jump inversion and SJT-hard sets}

\author{Rod Downey}
\address{School of Mathematics, Statistics and Operations Research, Victoria University of Wellington,
  Wellington, New Zealand}
\email{\href{mailto:downey@msor.vuw.ac.nz}%
{downey@msor.vuw.ac.nz}}
\urladdr{\url{http://homepages.mcs.vuw.ac.nz/~downey/}}

\author{Noam Greenberg}
\address{School of Mathematics, Statistics and Operations Research, Victoria University of Wellington,
  Wellington, New Zealand}
\email{\href{mailto:greenberg@msor.vuw.ac.nz}%
{greenberg@msor.vuw.ac.nz}}
\urladdr{\url{http://homepages.mcs.vuw.ac.nz/~greenberg/}}

\thanks{Both authors were supported by the Marsden Fund of New Zealand.}

\begin{document}

\begin{abstract}
	There are noncomputable c.e.\ sets, computable from every $\SJT$-hard c.e.\ set. This yields a natural pseudo-jump operator, increasing on all sets, which cannot be inverted back to a minimal pair or even avoiding an upper cone. 
\end{abstract}

\today
\maketitle

\section{Introduction}

While interactions between measure theory and computability theory can be traced
back to the 1950's through the work of de Leeuw, Moore, Shannon and Shapiro
\cite{deLeeuwMooreShannonShapiro} and Spector \cite{Spector:Hyprdegrees}, most of 
the spectacular development of these interactions
has really occurred in the last decade. 
Foremost in this development has been the use of methods from computability 
theory to understand and calibrate algorithmic randomness (see, for example,
Downey and Hirschfeldt \cite{DowneyHirschfeldtBook}, Nies \cite{Nies:Book} and Downey, Hirschfeldt, Nies and Terwijn \cite{DHNT:Callibrating}).
What has been less apparent is the extremely fruitful interaction the other 
way. That is, the use of algorithmic randomness to help us to understand computation. 
This paper is an example of this kind of interaction. 
We will use methods emanating from issues around $K$-triviality to answer 
a rather longstanding question in classical computability, explicitly 
articulated in Coles, Downey, Jockusch and LaForte \cite{ColesDowneyJockuschLaForte:CompletingPseudojump},
but going back to the papers of Jockusch and Shore \cite{JockuschShore:psudojump1,JockuschShore:pseudojump2}: we show that there is a natural pseudo-jump operator, increasing on all sets, which cannot be inverted while avoiding an upper cone. The techniques that we introduce in order to prove this result are far from what might have been tried  before the development of the randomness-related concepts of the last decade. We believe that they will have many other applications. 

The origins of the methods of this paper are attempts of  
combinatorial characterisations of the notion of $K$-triviality.
$K$-triviality has turned out to be a very important concept in algorithmic 
randomness. This concept originated in the work of Solovay \cite{Solovay:manuscript}, 
and was more recently developed starting with Downey, Hirschfeldt, Nies and Stephan \cite{DHNS:Trivial}. 
Although this notion is defined in terms of initial-segment complexity (a set is $K$-trivial if every initial segment of it is incompressible using a prefix-free machine),
% 
% $A$ is $K$-trivial if for all $n$, $K(A\rest n)\le^+ K(n)$\footnote{Here $K$ denotes prefix-free Kolmogorov complexity. A \emph{set}, or a \emph{real}, is an element of Cantor space $2^\w$. If $h,g\colon \w\to \w$, we say that $h\le^+ g$ if there is some $d<\w$ such that for all $n$, $h(n)\le g(n)+d$. We assume that the reader is familiar with the rudiments of algorithmic randomness as described in initial segments of Downey and Hirschfeldt \cite{DH}, Nies \cite{Niesbook}, or Li-Vitanyi \cite{LV}. We will be following the notation of Downey and Hirschfeldt \cite{DH}.}), 
the celebrated work of Nies, Hirschfeldt and others shows that $K$ triviality coincides with notions such as lowness for $K$, lowness for Martin-L{\"o}f randomness, lowness for weak 2-randomness, and being a base for randomness. All of these equivalent concepts express 
a lack of derandomisation power of an oracle with respect to some notion of algorithmic randomness: for example, a set $A\in 2^\w$ is low for Martin-L{\"o}f randomness if every set which is Martin-L{\"o}f random remains Martin-L{\"o}f random relative to $A$; in other words, $A$ cannot detect patterns in Martin-L\"{o}f random sets. We refer the reader to \cite{DowneyHirschfeldtBook,DHNT:Callibrating,Nies:Book,Nies:LownessRandomness,Nies:EliminatingConcepts} for details of such results.

One key question we might ask is whether there is some 
way to characterise classes like the $K$-trivials in terms 
of computability-theoretic considerations not involving randomness but ``purely discrete''
concepts like the Turing jump. In this context,
Terwijn \cite{Terwijn:PhD}, and then Terwijn and Zambella \cite{TerwijnZambella} found a new direction in this
programme.
They discovered that we could use what is called \emph{tracing} 
to give insight into such lowness concepts.
Tracing had its origins in set theory (see \cite{Raisonnier}), but in computability the concept is 
defined as follows.

\begin{definition}\label{def:trace}
	A \emph{trace} for a partial function $\psi\colon \w\to \w$ is a sequence $T=\seq{T(z)}_{z<\w}$ of finite sets such that for all $z\in \dom \psi$, $\psi(z)\in T(z)$. 
\end{definition}

Thus, a trace for a partial function $\psi$ indirectly specifies the values of $\psi$ by providing finitely many possibilities for each value; it provides a way of ``guessing'' the values of the function $\psi$. Such a trace is useful if it is easier to compute than the function $\psi$ itself. 
In some sense the notion of a trace is quite old in computability 
theory. W.\ Miller and Martin  \cite{MillerMartin:Hyperimmune} characterized the hyperimmune-free degrees
as those Turing degrees ${\ba}$ such that every (total) function $h\in {\ba}$ has a computable trace (the more familiar, but equivalent, formulation, is in terms of domination). In the same spirit, Terwijn and Zambella used a uniform version of hyperimmunity to characterise lowness for Schnorr randomness, thereby giving a 
``combinatorial'' characterisation of this lowness notion.

In this paper we are concerned not with how hard it is to compute a trace, but rather, how hard it is to enumerate it. 

\begin{definition}\label{def: c.e. trace} 
	A trace $T = \seq{T(z)}$ is \emph{computably enumerable in a Turing degree $\ba$} if the set of pairs $\left\{ (x,z)  \,:\, x\in T(z) \right\}$ is c.e.\ in $\ba$. 
\end{definition}
In other words, if uniformly in $z$, an oracle in $\ba$ can enumerate the elements of $T(z)$. It is guaranteed that each set $T(z)$ is finite, and yet if $T$ is merely c.e.\ in $\ba$, we do not expect $\ba$ to know when the enumeration of $T(z)$ ends. Thus, rather than using the exact size of each element of the trace, we use effective bounds on this size to indicate how strong a trace is: the fewer options for the value of a function, the closer we are to knowing what that value is. The bounds are known as order functions; they calibrate rates of growth of computable functions.  
	% 
	% 	Let $\seq{S_n}$ be a trace for a partial function $\psi$. 
	% 	\begin{enumerate}
	% 		\item The trace is \emph{computable} if the function $n\mapsto S_n$ is computable, that is, if there is a computable function $g$ such that for all $n$, $S_n = D_{g(n)}$, where $\seq{D_m}$ is an effective enumeration of all finite sets of numbers. 
	% 		\item The trace is \emph{c.e.} if the sequence $\seq{S_n}$ is uniformly c.e., that is, if there is a computable function $g$ such that for all $n$, $S_n = W_{g(n)}$, where $\seq{W_m}$ is an effective enumeration of all computably enumerable sets of numbers. 
	% 	\end{enumerate} 
	% \end{definition}
% 
% Thus ${\ba}$ is hyperimmune-free iff every function 
% $h\le_T {\ba}$ has a computable trace.

\begin{definition}\label{def:orders}
	An \emph{order function} is a nondecreasing, computable and unbounded function $h$ such that $h(0)>0$. If $h$ is an order function and $T=\seq{T(z)}$ is a trace, then we say that $T$ is an \emph{$h$-trace} (or that $T$ is \emph{bounded by $h$}) if for all $z$, $|T(z)|\le h(z)$. 
\end{definition}

In addition to measuring the sizes of c.e.\ traces, order functions are used to define uniform versions of traceability notions. For example, \emph{computable traceability}, the uniform version of hyperimmunity used by Terwijn and Zambella, is defined by requiring that traces for functions in a hyperimmune degree $\ba$ are all bounded by a single order function. 

% 
% Terwijn and Zambella showed that a real $A$ is low for Schnorr null tests
% % \footnote{Again for the present paper, it is not important what Schnorr randomness is, and it suffices to say that
% % it is a notion stronger than classical Martin-L\"of randomness. What is important is
% % that this is a combinatorial classification of a lowness concept 
% % without any recourse to concepts from effective randomness.} 
% if and only if there is some order function $h$ such that every (total) function computable from $A$ has a computable $h$-trace. This was later extended by Kjos-Hanssen, Nies and Stephan\cite{KNS} to show that such reals are exactly those that are low for Schnorr randomness. 
%  %In independent work, Ishmukhametov \cite{Ishmukhametov} used c.e.\ traceability (with a uniform bound) to construct strong minimal covers in the Turing degrees. 

Zambella (see Terwijn \cite{Terwijn:PhD}) observed that if $A$ is low for Martin-L\"of randomness then there is an order function $h$ such that every function computable from $A$ has a c.e.\ $h$-trace. This was improved by Nies \cite{Nies:LownessRandomness}, who showed that one can replace total by partial functions. In some sense it is natural to expect a connection between uniform traceability and lowness notions such as $K$-triviality; if every function computable (or partial computable) from $A$ has a c.e.\ $h$-trace, for some slow-growing order function $h$, then the value $\psi(n)$ of any such function can be described by $\log n + \log h(n)$ many bits.

Following this, it was a natural goal to characterise  $K$-triviality by tracing,  probably with respect to a family of order functions. This problem still remains open. However, an attempt toward a solution lead to the introduction of what seems now a fairly fundamental concept, which is not only interesting in its own right, but now has been shown to have deep connections with randomness.
\begin{definition}[Figuiera, Nies and Stephan \cite{FigueiraNiesStephan:SJT}] \label{def:sjt}
	Let $h$ be an order function. An oracle $A\in 2^\w$ is \emph{$h$-jump-traceable} if every $A$-partial computable function has a c.e.\ $h$-trace. An oracle is \emph{strongly jump-traceable} if it is $h$-jump-traceable for every order function $h$. 
\end{definition}

Figueira, Nies and Stephan gave 
a construction of a non-computable strongly jump-traceable c.e.\ set. Their construction bore a strong 
resemblance to the construction of a $K$-trivial c.e.\ set. J.\ Miller and Nies \cite{MillerNies:Open} asked if strong jump-traceability and $K$-triviality coincided.
Cholak and the authors \cite{CholakDowneyGreenberg:SJT1} showed that the strongly jump-traceable c.e.\ sets form a proper subclass of the c.e.\ $K$-trivial sets. They also showed that the class formed an ideal in the c.e.\ degrees. This ideal was later shown to be $\Pi_4^0$ complete by Ng \cite{Ng:SJT}, giving an alternative proof of the proper containment, as the $K$-trivial c.e.\ sets form a $\Sigma_3^0$ ideal. One of main contributions
of the  paper \cite{CholakDowneyGreenberg:SJT1} was the introduction of new combinatorial tools for dealing with the class of c.e., strongly jump-traceable sets, collectively known as  the ``box-promotion'' technique. 
We remark that recently, in  \cite{DGsjt2}
the authors showed how to adapt this technique to the non-c.e.\ case 
by showing that all strongly jump-traceable sets are $K$-trivial, when there was no
{a priori} reason that they should even be $\Delta_2^0$.

In view of these results it might seem that strong jump-traceability 
might be an interesting artifact of the studies of randomness, but 
as it turned out,
 the class of c.e., strongly jump-traceable sets has been shown to have remarkable connections with randomness. Greenberg, Hirschfeldt, and Nies \cite{GHN} proved that a c.e.\ set is strongly jump-traceable if and only if it is computable from every superlow random sets, if and only if it is computable from every superhigh random set; they related such random sets with the \emph{benign cost functions} which by work of Greenberg and Nies \cite{GreenbergNies:Benign} characterise c.e., strong jump-traceability. Greenberg and Turetsky \cite{GreenbergTuretsky:Demuth} complemented work of Ku{\v c}era and Nies \cite{KuceraNies:Demuth} and showed that a c.e.\ set is strongly jump-traceable if and only if it is computable from a Demuth random set, thus solving the Demuth analogue of the random covering problem, which remains open for Martin-L\"{o}f randomness and $K$-triviality. Other attractive spin-offs in the arena of randomness include Nies's new work on the calculus of cost functions \cite{Ncost}. This material is only just beginning to work itself out and we expect a lot more to grow from these ideas.

\

The goal of the present paper is to use strong jump-traceability to 
solve a longstanding question in classical computability. This question concerns 
what are called {pseudo-jump operators.}

\begin{definition}[Jockusch and Shore \cite{JockuschShore:psudojump1,JockuschShore:pseudojump2}]
A \emph{pseudo-jump operator} is a map $J\colon 2^\w\to 2^\w$ of the form $J(A) = A\oplus W^A_e$.\footnote{Here $W_e^A$ is the $e\tth$ c.e.\ set relativised to $A$, in some fixed list of all c.e.\ 
operators.} A pseudo-jump operator $J$ is \emph{everywhere increasing} if for all $A\in 2^\w$, $J(A)>_\Tur A$. \end{definition}

The most natural pseudo-jump operator is the Turing jump operator, which maps each set $A\in 2^\w$ to $A'$, the halting set relative to $A$. Jockusch and Shore's key insight was that this generalisation of the Turing jump allowed one to give simple constructions of Turing degrees of various prescribed properties. The key concept was \emph{pseudo-jump inversion}.

\begin{theorem}[Jockusch and Shore \cite{JockuschShore:psudojump1}]\label{thm:pseudojump_inversion}
	Let $J$ be a pseudo-jump operator. Then there is some non-computable c.e.\ set $A$ such that $J(A)\equiv_\Tur \cero'$. 
\end{theorem}

The degree $\ba$ of the c.e.\ set $A$ given by Theorem \ref{thm:pseudojump_inversion} is an instance of \emph{inverting} the operator $J$. Roughly, the idea is that $J$ explains how to relativise to any oracle the construction of a c.e.\ set $J(\emptyset)$. Inverting the operator $J$ allows us, up to Turing equivalence, to view the halting problem $\emptyset'$ as the result of the construction $J$, relativised to some c.e.\ oracle. For example, applying pseudo-jump inversion to the Turing jump operator gives a non-computable \emph{low} set, a set whose Turing jump is as simple as possible. In turn, inverting the construction of a low set yields a \emph{high} set, an incomplete c.e.\ set whose Turing jump is $\cero''$ (as high as possible for a c.e.\ set). Jockusch and Shore went on to give simple constructions of c.e.\ degrees in every level of the low$_n$ and high$_n$ hierarchy using pseudo-jump inversion. 
These methods have seen many generalisations, and have extensions to 
randomness, as witnessed by Nies \cite[Theorem 6.3.9]{Nies:Book}, and to $\Pi^0_1$ classes
by Cenzer, LaForte and Wu \cite{CenzerLaforteWu:Pseudojump}, in some sense extending the Jockusch-Soare low basis theorem. 

In spite 
of the usefulness of pseudo-jump operators, there is distinct lack of general theory
concerning them, aside from the original Jockusch-Shore papers. 
Coles, Downey,  Jockusch and LaForte 
\cite{ColesDowneyJockuschLaForte:CompletingPseudojump} studied the general theory of these operators, trying to understand what techniques pseudo-jump inversion would be compatible with.
Such questions were implicit in the Jockusch-Shore papers.
For example,  in \cite{ColesDowneyJockuschLaForte:CompletingPseudojump} it is shown that pseudo-jump inversion is compatible with a Friedberg strategy: for any everywhere increasing
pseudo-jump operator $J$, it 
is possible to construct two Turing incomparable c.e.\ inversions of $J$.  However, the question of whether pseudo-jump inversion is compatible with avoiding upper cones, and even with the construction of a minimal pair, remained open. In \cite{ColesDowneyJockuschLaForte:CompletingPseudojump}, the authors construct a pseudo-jump operator $J$ which is increasing on c.e.\ sets, for which inversion cannot be combined with upper cone avoidance. That is, there is a non-computable c.e.\ set $E$ which is computable from all c.e.\ inversions of $J$. However, there is no reason to believe that this operator $J$ is increasing on all sets, rather than only on the c.e.\ ones. The difficulty of making $J$ increasing globally is similar to the problem of producing a degree-invariant solution for Post's problem. Moreover, this operator $J$ is unnatural, in that it is given by a direct priority construction. Their construction was a $\cero'''$-priority argument.

In this paper we solve this question, by showing that there is an everywhere-increasing pseudo-jump operator for which inversion cannot be combined with upper-cone avoidance. This operator is the relativisation of the construction of a non-computable strongly jump-traceable set from \cite{FigueiraNiesStephan:SJT}. Thus our example is natural. While our construction is combinatorially complex, it does not use the $\cero'''$-priority machinery utilised in \cite{ColesDowneyJockuschLaForte:CompletingPseudojump}, and so is \emph{logically} simpler. 

Relativising this construction, let $J_{\SJT}$ be a pseudo-jump operator, everywhere increasing, such that for all $A\in 2^\w$, $J_{\SJT}(A)$ is strongly jump-traceable relative to $A$. Because every strongly jump-traceable set is ``very low'' (it is $K$-trivial and so superlow), every inversion $A$ of $J_{\SJT}$ must be ``very high'': $\emptyset'$ is $K$-trivial relative to $A$, and so $A$ is superhigh. Inversions of $J_{\SJT}$, namely c.e.\ sets relative to which $\emptyset'$ is strongly jump-traceable, were first studied by Ng in \cite{Ng_very_high}, where they are called \emph{ultrahigh}. 

Nies related notions of lowness, such as superlowness, $K$-triviality and jump-traceability, to so-called \emph{weak reducibilities}. These are partial relativisations of these lowness notions which are made so as to obtain transitive relations on $2^\w$. The best-known weak reducibility $\le_{\LR}$ is obtained by partially relativising lowness for randomness, equivalently lowness for $K$. The weak reducibility corresponding to strong jump-traceability, $\le_{\SJT}$, is obtained by relativising the complexity of the traces, but by preserving the complexity of the bounds on the traces:

\begin{definition}\label{def:SJT_reducibility}
	For $A,B\in 2^\w$, we let $A\le_{\SJT} B$ if for every (computable) order function $h$, every $B$-partial computable function $\psi$ has an $A$-c.e.\ trace bounded by $h$. 
\end{definition}

Akin to other reducibilities, we say that a set $B\in 2^\w$ is \emph{$\SJT$-hard} if $\emptyset'\le_{\SJT} B$. That is, if for every order function $h$, every partial $\Sigma^0_2$ function has an $A$-c.e.\ trace bounded by $h$. Certainly every ultrahigh set is SJT-hard. 

The main theorem of this paper is:

\begin{theorem}\label{thm_main}
	There is a noncomputable c.e.\ set which is computable from every $\SJT$-hard c.e.\ set. 
\end{theorem}

Applying the result to ultrahigh sets and so to inversions of $J_\SJT$, we get:

\begin{corollary}\label{cor_pseudojump}
	There is a pseudojump operator $J$, increasing on all sets, which cannot be inverted while avoiding any prescribed upper cone.
\end{corollary}

A question pursued by several researchers is whether there is a minimal pair of $\LR$-hard c.e.\ degrees. By a relativisation of Nies's results, a c.e.\ degree $\ba$ is $\LR$-hard if and only if $\emptyset'$ is $K$-trivial relative to $\ba$. The interest in LR-hard degrees was sparked by work of Kjos-Hanssen, J.\ Miller and Solomon \cite{K-HMS}, who showed that a Turing degree is LR-hard if and only if it is almost everywhere dominating, a notion suggested by Dobrinen and Simpson \cite{Dobrinen.Simpson:04}. Relativising a result of Nies's mentioned above, we see that there is an order function $h$ such that for any $\LR$-hard c.e.\ degree $\ba$, $\cero'$ is $h$-jump-traceable relative to $\ba$. An examination of the proof of Theorem \ref{thm_main} reveals that in fact, there is an order function $g$ and a non-computable c.e.\ set $E$ which is computable from every c.e.\ set relative to which $\cero'$ is $g$-jump-traceable. If we could make $g$ grow at least as quickly as $h$, we would settle the question in the negative. Currently we know that for $h$, we can take any order function such that $\sum 2^{-h(n)}$ is computable and finite; the proof of Theorem \ref{thm_main} gives an order function $g$ whose growth rate is roughly $\log \log n$. The gap does not seem too large, and at least one of the authors believes that it can be bridged. 

% 
% 
% 
% 
% We remark that the proof of Theorem \ref{1} is again a delicate box 
% promotion strategy. This is the first time this strategy has been used 
% as a highness strategy, and likely it will promise further applications. 
% It is also notable that all previous attempts 
% so solving the cone avoidance with direct constructions failed 
% miserably. We could speculate that perhaps methods like ours might
% be of relevance to other questions about degree invariance in the degrees.

As mentioned above, mixing randomness with Turing reducibility has resulted in interesting classes of c.e.\ degrees. The collection of c.e.\ degrees which lie below all $\SJT$-hard c.e.\ degrees seems to be a new ideal of c.e.\ degrees, about which, at this point, we know next to nothing. Further research is definitely needed. This would no doubt require a further refinement of the box-promotion technique, which is first used to deal with highness notions in this paper.

\section{Minimal pairs}

The proof of Theorem \ref{thm_main} is technical. As a warm-up, we prove a weaker result first: that there is no minimal pair of c.e.\ SJT-hard degrees. 

\subsection{Discussion}

Let $A^0$ and $A^1$ be SJT-hard c.e.\ sets. We enumerate a set $E$, which we make noncomputable and reducible to both $A^0$ and $A^1$. The noncomputability requirements are the familiar 
\begin{description}
	\item[$P^e$] $E\ne \vphi^e$,  
\end{description}
where $\seq{\vphi^e}$ is an effective enumeration of all partial computable functions. These requirements are met by the Friedberg-Muchnik strategy: a requirement $P^e$ appoints a \emph{follower} $x$, waits for the follower to be \emph{realised}, which means $\vphi^e(x)\converge =0$, and then wants to enumerate $x$ into $E$. 

Of course, to ensure that $E$ is computable from $A^0$ and from $A^1$, when the requirement $P^e$ appoints $x$, it needs to determine \emph{uses} $u^0$ and $u^1$ for reducing the question ``$x\in E$?'' to $A^0$ and to $A^1$. We are not allowed to enumerate $x$ into $E$ unless \emph{both} $A^0$ and $A^1$ change below $u^0$ and $u^1$ respectively. Moreover, if, for example $A^0\rest{u^0}$ changes at some stage $s$, we would either need to get a change in $A^1\rest{u^1}$ more or less at this stage, and then we would have permission to put $x$ into $E$; or we would need to reset $u^0$ and wait for another $A^0$ change on the new use.
That is, permissions need to be more or less simultaneous. Furthermore, whilst $u^0$ and $u^1$ are chosen to be large when $x$ is appointed, a long time passes between that stage and the stage at which $x$ is realised; of course, before $x$ is realised, we do not want to enumerate it into $E$. Hence, when we want to enumerate $x$ into $E$, the uses $u^0$ and $u^1$ would be relatively small, and so voluntary $A^0$ and $A^1$ changes below their uses are unlikely.

Fortunately, the fact that $A^0$ and $A^1$ are SJT-hard means that they do have to change often; for example, they are both high. The SJT-hardness gives us a mechanism for forcing desirable changes in these sets. SJT-hardness means that for both $i=0,1$, the set $A^i$ can trace 
any $\Sigma_2^0$ partial function $\psi$ by a trace $T^i = \seq{T^i(z)}_{z<\w}$ which is bounded by any prescribed computable order function $h$. That is, if $z\in \dom \psi$, then $\psi(z)\in T^i(z)$, and for all $z$, $|T^i(z)| \le h(z)$. Now a $\Sigma^0_2$ function $\psi$ is the partial limit of a sequence $\seq{\psi_s}_{s<\w}$ of uniformly computable functions; $\psi(z)=d$ if and only if for almost all $s$, $\psi_s(z) = d$. The trace $T^i$ is uniformly c.e.\ in $A^i$, that is, every $d\in T^i(z)$ is enumerated into $T^i(z)$ with some $A^i$-use, which we denote by $u^i(z,d)$. At stage $s$, we let $T^i_s(z)$ be the collection of numbers enumerated into $T^i(z)$ by stage $s$ with the oracle $A^i_s$, the collection of numbers already in $A^i$ at stage $s$. An important point is that if $\psi_s(z)=d$ at some stage $s$, then either $d\in T^i(z)$, or there is some stage $t>s$ such that $\psi_t(z)\ne d$. Thus, if we are defining $\psi$, then when we let $\psi_s(z)=d$, we can commit to never changing the value $\psi_t(z)$ away from $d$ unless we see that $d\in T^i_t(z)$. This implies that for all $s$ and $z$, there is some $t>s$ such that $\psi_s(z)\in T^i_t(z)$. If $d = \psi_s(z)$ is an element of $T^i_t(z)$, with $A^i$-use $u= u^i_t(z,d)$, then at a stage $r>t$, either $d$ is still an element of $T^i_r(z)$, or $A^i_t\rest {u}\ne A^i_r\rest{u}$. In other words, extracting $d$ from $T^i(z)$ requires an $A^i$-change below $u$. In particular, this means that if at stage $t$ (or at a later stage) we change $\psi_t(z)$ to have value different from $d$, then either this new value needs to be enumerated into $T^i(z)$ in addition to the old value $d$, or the old value needs to be extracted from $T^i(z)$ by an $A^i$-change below $u$. If we repeat this process more than $h(z)$ many times, then as $T^i(z)$ cannot contain $h(z)+1$ many distinct possible values for $\psi(z)$, at some cycle we will have forced a change in $A^i$. The smaller the value $h(z)$ is, the fewer cycles we need to go through before we obtain the desired $A^i$-change. 
% 
% 
% That is, if $\varphi(x,s)=d$ and $k(x)=n$, say, then 
% for each $A^i$,
% $d$ must occur in the box tracing $\varphi(x)$ 
% as a $\Sigma_1^{A^i}$ fact. This means it will appear with some $A^i$-use,
% and for it to leave the box $A^i$ would need to change at a later stage 
% below that use. 
% In particular, if $\varphi(x,t)$ 
% changes for $t>s$ from $d$ to some new $d'$, then $d'$ must enter the 
% $A^i$ box devoted to $x$, and hence \emph{either} $A^i$ must change 
% on the $d$-use, or $d'$ must enter the $x$-box with a new use, and occupying 
% another of the $n$ possible locations. 

The construction below is a balancing act. We will show that if $h$ grows sufficiently slowly, then we will be able to force sufficiently many simultaneous $A^0$- and $A^1$-changes, to get enough simultaneous permissions so that for every requirement $P^e$ there is \emph{some} follower $x$ for $P^e$ which is permitted to enter $E$ (or is never realised). So as described above, we define a $\Sigma^0_2$ partial function $\psi$, by letting $\psi$ be the partial limit of a uniformly computable sequence $\seq{\psi_s}$ of functions. We also define an order function $h$. As we argue below, the recursion theorem will give us, for both $i=0,1$, $A^i$-c.e.\ traces $T^i = \seq{T^i(z)}_{z<\w}$ for $\psi$ which are both bounded by $h$. As described above, we let $T^i_s(z)$ be the stage $s$-approximation for $T^i(z)$, and we assume that for all $s$ and $z$ there is some $t>s$ such that $\psi_s(z)\in T^i_t(z)$. 

The general process for obtaining an $A^i$-permission for some follower $x$ for $P^e$ follows the procedure, described above, of manufacturing an $A^i$-change. We associate some input $z$ with the requirement $P^e$. At some stage $s_0$, we appoint a first follower $x_0$ for $P^e$, and define $\psi_{s_0}(z) = s_0$. We then wait for a stage $t_0>s_0$ at which we see that $s_0$ is enumerated into $T^i(z)$. This enumeration yields an $A^i$-use $u^i_{t_0}(z,s_0)$; we determine that the $A^i$-use $v^i(x_0)$ which is used for the reduction of the question ``$x_0\in E$?'' to $A^i$ equals the use $u^i_{t_0}(z,s_0)$. Now suppose that at some stage $s_1>t_0$, the follower $x_0$ is realised; so for $A^i$-permission for $x_0$, we need to see a future change in $A^i\rest{v^i(x_0)}$. The definition of $v^i(x_0)$ means that if $A^i$ does not give permission to $x_0$ by a stage $r$, then $s_0\in T^i_r(z)$. We then appoint a new follower $x_1$, and redefine $\psi_{s_1}(z)$ to equal $s_1$. We wait for a stage $t_1>s_1$ at which we see that $s_1$ is enumerated into $T^i(z)$. Now there are two possibilities: if $A^i_{t_1}\rest{v^i(x_0)}\ne A^i_{t_0}\rest{v^i(x_0)}$, then $x_0$ receives $A^i$-permission as required. Otherwise, we define the use $v^i(x_1) = u^i_{t_1}(z,s_1)$ for reducing the question ``$x_1\in E$?'' to $A^i$; we wait for another stage $s_2$ at which $x_1$ is realised, and repeat the process: appoint a new follower $x_2$, change $\psi_{s_2}(z)$ to equal $s_2$, find a stage $t_2>s_2$ at which $s_2$ is enumerated into $T^i(z)$. At stage $t_2$, if neither $x_0$ nor $x_1$ are permitted by $A^i$, then we appoint a new follower and repeat the cycle. As we mentioned above, this process cannot repeat more than $h(z)+1$ many times, and so eventually, either we appoint a follower which never gets realised, or some realised follower is permitted by $A^i$. 

The reader may ask: after all, we define $h$. Why do we not just define $h(z) = 1$? In this case, we do not need to appoint more than one follower, as the first follower appointed is guaranteed to receive permission. There are two reasons. First, while we define $h$, the recursion theorem levies an ``overhead'', which means that it gives us a constant $c\ge 1$ such that $T^i$ is only bounded by $\max\{c,h\}$ and not by $h$. The second reason is more important. While we define $h$, we need to ensure that $h$ is unbounded. That means that for any value $k$, we may only define $h(z)=k$ for finitely many inputs $z$. Fixing $z$, if $x$ is a follower for $P^e$ which is never realised, we will not need permission from $A^i$ for $x$, but we also need to ensure that the use of reducing $x\in E$ to $A^i$ is bounded. If we let other followers, for other requirements, access $z$ by changing $\psi(z)$ and thus the $A^i$-use of enumerating $\psi(z)$ into $T^i(z)$, then the use of reducing $x\in E$ to $A^i$, which has to be tied to this use (to enable permission if $x$ does get realised), will increase each time $\psi(z)$ is changed. This should not happen infinitely often. And so, for yet unrealised followers, different requirements $P^e$ need to access different inputs $z$, and so the values $h(z)$ for the inputs $z$ cannot be bounded. 

\medskip

The discussion so far shows how to obtain permission from a single set $A^0$ or $A^1$, and so would suffice if we wanted to show that both sets are not computable. However, to enumerate a follower $x$ into $E$, we need permission from both sets $A^0$ and $A^1$. If we did have ``1-boxes'', namely access to inputs $z$ such that $|T^i(z)|$ is known to be bounded by 1, then we could run the process above on both sides: appoint a follower $x$ at a stage $s$; define $\psi_s(z)=s$; for both $i<2$, tie the $A^i$-use of $x$ to the $A^i$-use $v^i(x) = u^i_{t}(z,s)$ of enumerating $s$ into $T^i(z)$ at some stage $t>s$; if $x$ is realised at some $s'>t$, redefine $\psi(z)=s'$, get permission for $x$ from both $A^0$ and $A^1$, and enumerate $x$ into $E$. The point here is that 
if we have 1-boxes then the \emph{only} way a new value can enter the box is that the relevant $A^i$ changes on the use.
However, we may not have 1-boxes. If $h(z)>1$, and we attempt to run the process on both sides, we may be faced with the following situation
which is caused by certain \emph{timing} considerations. At stage $s_0$, we appoint the follower $x_0$ and define $\psi_{s_0}(z)=s_0$;  at stage $t_0>s_0$, we associate with $x_0$ the $A^i$-uses $v^i(x_0) = u^i_{t_0}(z,s_0)$. At stage $s_1>t_0$, $x_0$ is realised; we appoint a new follower $x_1$, set $\psi_{s_1}(z)=s_1$, and wait for $s_1$ to be enumerated into both $T^0(z)$ and $T^1(z)$ at a stage $t_1>s_1$. At stage $t_1$, we see that $A^0$ permits $x_0$, but $A^1$ does not. We are not allowed to enumerate $x$ into $E$ at stage $t_1$. But we are also not allowed to \emph{keep the $A^0$-permission for $x$ open}: at stage $t_1$, we must associate $x_0$ with a new $A^0$-use, or we will not be able to compute the answer to the question ``$x\in E$?'' from $A^0$ alone. We can do this; we redefine $u^0(x_0) = u^0(x_1) = u^0_{t_1}(z,s_1)$. Say that at a later stage $s_2$, $x_1$ is realised as well. Suppose, for simplicity, that $h(z)=2$. Thus if we change $\psi(z)$ at stage $s_2$, then we are guaranteed that $A^1$ permits either $x_0$ or $x_1$ at some $t_2>s_2$. But the $A^0$-permission for $x_0$ at stage $t_1$ means that $s_0$ is extracted from $T^0(z)$, and so $T^0_{t_1}(z)$ might contain but one element, $s_1$. In this case, $A^0$ may refuse to permit either $x_0$ or $x_1$ at stage $t_2$. The permission that $A^1$ gives at stage $t_2$, though, gives $A^1$ the opportunity to empty $T^1(z)$ and leave it with one element, $s_2$. This strategy allows $A^0$ and $A^1$ together to keep see-sawing, always giving permission on one side, and denying it on the other. 
Hence whilst the opponent gives permissions on both sides, because they are 
out of phase, we cannot use them for enumeration of $x_0$, or indeed any other follower, into $E$.

In fact, if the growth rate of the order $h$ is $2^n$, then we know that
it is possible to construct a minimal pair of c.e.\ sets $B^1,B^2$,
which are superhigh with truth table bound  $2^n$. (See for example
Ng's \cite{Ng_very_high}). Thus somehow we will need to have a more slowly growing 
$h$, and this is ``dually'' similar to the problems in showing that
the strongly jump-traceable c.e.\ sets form an ideal, as shown in
Cholak, Downey and Greenberg \cite{CholakDowneyGreenberg:SJT1}.

We need to break the symmetry between $A^0$ and $A^1$. Using a 
slowly growing order function $h$, this is done by an inductive process known as ``box promotion''. This promotion has already been exploited. If some realised follower $x$ is tied to an $A^i$-box $\{z\}$ via some value $s\in T^i(z)$ (and use $v^i(x)=u_t^i(z,s)$), we change $\psi(z)$ away from $s$, and $A^i\rest{u}$ does not change, then the $h(z)$-box $\{z\}$ has been \emph{promoted} to being an $h(z)-1$-box. Other followers, weaker than $x$, can use the box, imagining that it only has $h(z)-1$ many slots left to fill. This promotion can be reversed only if $A^i$ relents and permits $x$. 

We sketch how this works. Suppose, as first approximation, that the constant $c$ supplied by the recursion theorem equals 1, that is, we have access to 1-boxes. Since we only have finitely many 1-boxes, we consider a requirement $P^e$ which initially only has access to 2-boxes. Suppose that $x$ is a realised follower for $P^e$; at some stage $s$, the $A^i$-uses are $v^i(x) = u^i_t(z_i,t)$, where $\{z_0\}$ and $\{z_1\}$ are 2-boxes. We ask for $A_1$-permission for $x$ by changing $\psi(z_1)$ away from $t$. While permission is denied, $\{z_1\}$ is effectively a 1-box, and other followers may use it to define uses and obtain permission from $A_1$. Since we do not want to drive the use $v^0(x)$ to infinity, other followers, for the time being, are not allowed to change $\psi(z_0)$. Once $A^1$-permission is received, we move the $A^1,x$-pointer from $z_1$ to a 1-box $\{w\}$, and obtain a new version of $v^1(x)$, based on the use of enumerating a new value $\psi(w)$ into $T^1(w)$. We then ask for $A^0$-permission for $x$, making the 2-box $\{z_0\}$ ``active'': while $A^0$-permission for $x$ is denied, $\{z_0\}$ is effectively a 2-box for other followers, who may now change $\psi(z_0)$ with impunity, defining uses based on $z_0$ and obtaining permissions from $A^0$. Once $A^0$-permission for $x$ is obtained, we can change $\psi(w)$ and get a guaranteed $A^1$-permission for $x$ as well, and enumerate $x$ into $E$. See Figures \ref{fig1} and \ref{fig2}.

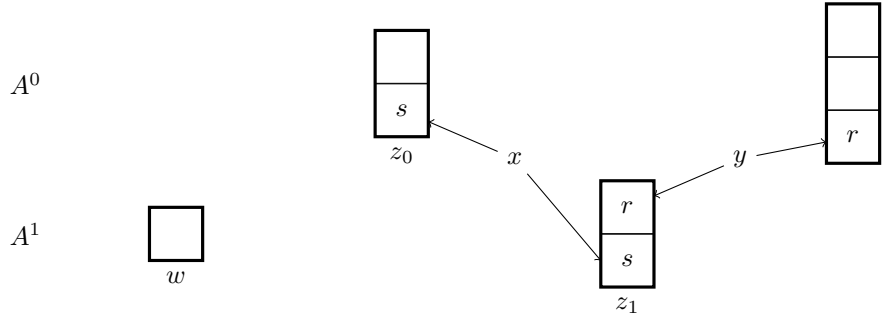
\begin{figure}[h]

	\begin{center}
		\begin{tikzpicture}
			[box/.style = {draw,very thick,minimum size=7mm,nodes={draw,very thin},inner sep=0mm}]
			
			\matrix[box] at (9,2) { \node {};  \\ \node {};  \\ \node (r_0) {$r$}; \\ };			
			\matrix[box,label=below:$z_0$] (z_0) at (3,2) { \node {};  \\ \node (s_0) {$s$}; \\ };
			\matrix[box,label=below:$z_1$] (z_1) at (6,0) { \node (r_1) {$r$};  \\ \node (s_1) {$s$}; \\ };
			\matrix[box,label=below:$w$] (w)   at (0,0) { \node {}; \\ };
			
			\node at (-2,0) {$A^1$};
			\node at (-2,2) {$A^0$};
			
			\node at (4.5,1) {$x$} 
				edge [->] (s_1.west)
				edge [->] (s_0);

			\node at (7.5,1) {$y$}
				edge [->] (r_0)
				edge [->] (r_1);
			
		\end{tikzpicture}
	\end{center}

	\caption{The follower $x$ is waiting for an $A^1$-permission. Meanwhile, follower $y$ can treat $\{z_1\}$ as a 1-box.}
	\label{fig1}
\end{figure}

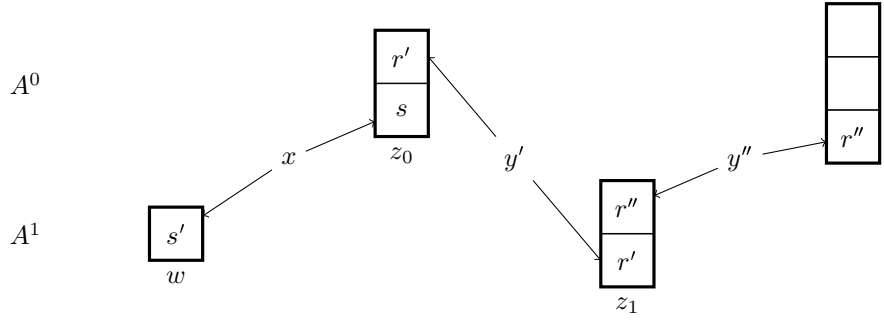
\begin{figure}[h]

	\begin{center}
		\begin{tikzpicture}
			[box/.style = {draw,very thick,minimum size=7mm,nodes={draw,very thin},inner sep=0mm}]
			
			\matrix[box,label=below:$w$] (w)   at (0,0) { \node (s_1) {$s'$}; \\ };
			\matrix[box,label=below:$z_0$] (z_0) at (3,2) { \node (r_0) {$r'$};  \\ \node (s_0) {$s$}; \\ };
			\matrix[box,label=below:$z_1$] (z_1) at (6,0) { \node (r'_1) {$r''$};  \\ \node (r_1) {$r'$}; \\};
			\matrix[box] at (9,2) { \node {};  \\ \node {};  \\ \node (r'_0) {$r''$}; \\ };			
						
			\node at (-2,0) {$A^1$};
			\node at (-2,2) {$A^0$};
			
			\node at (1.5,1) {$x$} 
				edge [->] (s_1)
				edge [->] (s_0);

			\node at (4.5,1) {$y'$}
				edge [->] (r_0.east)
				edge [->] (r_1.west);

			\node at (7.5,1) {$y''$}
				edge [->] (r'_0)
				edge [->] (r'_1);
			
		\end{tikzpicture}
	\end{center}

	\caption{Follower $x$ received permission from $A^1$. Its $A^1$-pointer moved to $w$. It is now waiting for $A^0$-permission. Meanwhile, follower $y'$ is treating $\{z_0\}$ as a 1-box, and follower $y''$ is treating $\{z_1\}$ as a 1-box.}
	\label{fig2}
\end{figure}

Now, an examination of what happens further to the right explains why we need to use \emph{metaboxes}. A metabox is an aggregate of boxes of some fixed size. It is used to amplify the promotion given by a refusal of some set to give permission. Suppose again, for example, that both $A^0$- and $A^1$-uses of a follower $x$ are associated with 2-boxes $\{z_0\}$ and $\{z_1\}$ as above. While waiting for $A^1$-permission, other followers treat $\{z_1\}$ as a 1-box. Such followers point to $\{z_1\}$ and to an $A^0-3$-box $\{z_2\}$. We need at least two such followers, $y_0$ and $y_1$, to turn $\{z_2\}$ into a 1-box. However, while waiting for permission from $A^0$, the $A^1$-use of $y_0$ and $y_1$, which is tied to $\{z_1\}$, should not change. This means that $y_0$ and $y_1$ should not be sharing $z_1$; they should have different versions of $z_1$, each for its own use. [The reader may ask: if there are only two followers, $y_0$ and $y_1$, then not much harm will come from changing, say, $y_0$'s use when $y_1$ acts, as $y_1$ will act only once? But it is possible that $y_0$ is fixed, but infinitely many followers playing the role of $y_1$ come and go, each acting once. This is why $y_0$ needs a box of its own on the $A^1$ side.] The solution is to tie $x$'s use to to a single input $z_1$, but to a metabox $M$ of inputs. When the use is set up, at stage $s$ say, we define $\psi(z)=s$ for all $z\in M$. We wait for $s$ to show up in $T^1(z)$ for all $z\in M$. We then let the use $v^1(x)$ be the maximum of the uses $u^1(z,s)$ of enumerating $s$ into $T^1(z)$ for $z\in M$. While $A^1$ does not permit $x$, every $z\in M$ can be treated as a 1-box. This means that we can then split $M$ into two, one part tied to $y_0$ and another to $y_1$. See Figure \ref{fig_metabox}.

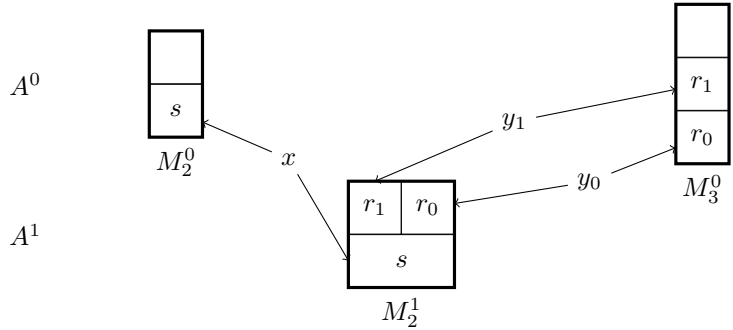
\begin{figure}[h]

	\begin{center}
		\begin{tikzpicture}
			[box/.style = {draw,very thick,minimum size=7mm,nodes={draw,very thin},inner sep=0mm},
			bbox/.style = {draw,very thin,minimum size=7mm,inner sep=0mm,on grid}]

			\matrix[box,label=below:$M^0_2$] (M_02) at (0,2) { \node{}; \\ \node (s_0) {$s$};  \\ };
			
			\node[bbox,anchor=south east] (r_10) at (3,0) {$r_1$};
			\node[bbox,right=7mm of r_10] (r_11) {$r_0$};
			\node[bbox,below=0mm of r_10.south east,minimum width=14mm] (s_1) {$s$};
			\node[draw,very thick,fit=(r_10) (r_11) (s_1),inner sep=-0mm,label=below:$M^1_2$] {};

			\matrix[box,label=below:$M^0_3$] at (7,2) { 
				\node {};  \\ \node (r_00) {$r_1$};  \\ \node (r_01) {$r_0$}; \\ };			
						
			\node at (-2,0) {$A^1$};
			\node at (-2,2) {$A^0$};
			
			\node at (1.5,1) {$x$} 
				edge [->] (s_1.west)
				edge [->] (s_0);

			\node at (4.5,1.5) {$y_1$}
				edge [->] (r_00)
				edge [->] (r_10.north);
					
			\node at (5.5,0.7) {$y_0$}
				edge [->] (r_01)
				edge [->] (r_11);
								
		\end{tikzpicture}
	\end{center}

	\caption{The metabox $M^1_2$ is promoted while $x$ is waiting for $A^1$-permission. It is split between $y_0$ and $y_1$; while both $y_0$ and $y_1$ wait for $A^0$-permission, other followers can treat $M^0_3$ as a 1-box and get immediate permission when required.}
	\label{fig_metabox}
\end{figure}

Of course, we need to apply this reasoning to every level: thinking toward the 4-boxes, we may need to split the collection of 3-boxes into at least three disjoint parts. But it is not sufficient to have only $k$ many $k$-boxes.\footnote{We remark at this point that it would be \emph{extremely nice} to have a version of this construction requiring only $k$ many $k$-boxes. As indicated in the introduction, we believe that this would lead to a proof that there is no minimal pair of LR-hard c.e.\ sets.} Consider again the case of 2-boxes. Even before any promotions are made to 2-boxes, we may need disjoint 2-boxes (on the $A^1$ side) for two followers $y_0$ and $y_1$ which are waiting for permission from $A^0$ on 3-boxes. Either one of these may need to be split up in the future: at most one, but we cannot tell in advance which one. Suppose that $y_0$ is stronger than $y_1$. If $y_0$ gets permission from $A^0$, and its $A^0$ pointer moves from $M^0_3$ to $M^0_2$, then $y_1$ gets cancelled, the part of $M^1_2$ which was pointed to by $y_0$ gets promoted, and it will be split up between future followers pointing at $M^1_2$ and $M^0_3$. The part of $M^1_2$ which was pointed to by $y_1$ is not promoted, and will not be used until $y_0$ is cancelled or moved. If $y_1$, but not $y_0$, gets permission, then $y_0$ is not cancelled, but the part of $M^1_2$ to which it is pointed is not promoted, and will not be used by other followers; a future follower will point to sub-boxes of the metabox pointed to by $y_1$. See Figures \ref{fig3}, \ref{fig4}, and \ref{fig5}. So we need at least three 2-boxes, and in general about $k^k$-many $k$-boxes.

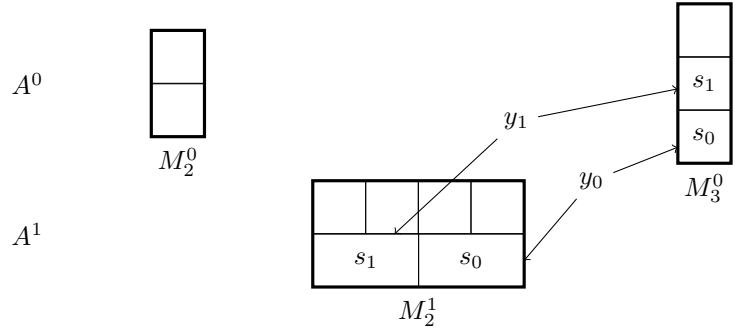
\begin{figure}[h]

	\begin{center}
		\begin{tikzpicture}
			[box/.style = {draw,very thick,minimum size=7mm,nodes={draw,very thin},inner sep=0mm},
			bbox/.style = {draw,very thin,minimum size=7mm,inner sep=0mm,on grid}]

			\matrix[box,label=below:$M^0_2$] (M_02) at (0,2) { \node{}; \\ \node{};  \\ };
			
			\node[bbox,anchor=south east] (r_10) at (2.5,0) {};
			\node[bbox,right=7mm of r_10] (r_11) {};
			\node[bbox,right=7mm of r_11] (r_00) {};
			\node[bbox,right=7mm of r_00] (r_01) {};
			\node[bbox,below=0mm of r_10.south east,minimum width=14mm] (s_1) {$s_1$};
			\node[bbox,below=0mm of r_00.south east,minimum width=14mm] (s_0) {$s_0$};
			\node[draw,very thick,fit=(r_10) (r_11) (s_1) (s_0) (r_00) (r_11),inner sep=-0mm,label=below:$M^1_2$] {};

			\matrix[box,label=below:$M^0_3$] at (7,2) { 
				\node {};  \\ \node (s_0') {$s_1$};  \\ \node (s_1') {$s_0$}; \\ };			
						
			\node at (-2,0) {$A^1$};
			\node at (-2,2) {$A^0$};
			
			\node at (4.5,1.5) {$y_1$}
				edge [->] (s_1)
				edge [->] (s_0');
					
			\node at (5.5,0.7) {$y_0$}
				edge [->] (s_0.east)
				edge [->] (s_1');
								
		\end{tikzpicture}
	\end{center}
	\caption{Both $y_0$ and $y_1$ seek permission from $A^0$. We do not know which, if either, will be permitted.}
	\label{fig3}
\end{figure}

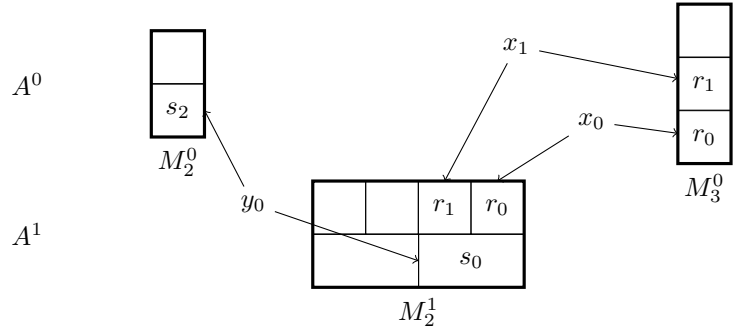
\begin{figure}[h]

	\begin{center}
		\begin{tikzpicture}
			[box/.style = {draw,very thick,minimum size=7mm,nodes={draw,very thin},inner sep=0mm},
			bbox/.style = {draw,very thin,minimum size=7mm,inner sep=0mm,on grid}]

			\matrix[box,label=below:$M^0_2$] (M_02) at (0,2) { \node{}; \\ \node (s_2) {$s_2$};  \\ };
			
			\node[bbox,anchor=south east] (r_10) at (2.5,0) {};
			\node[bbox,right=7mm of r_10] (r_11) {};
			\node[bbox,right=7mm of r_11] (r_00) {$r_1$};
			\node[bbox,right=7mm of r_00] (r_01) {$r_0$};
			\node[bbox,below=0mm of r_10.south east,minimum width=14mm] (s_1) {};
			\node[bbox,below=0mm of r_00.south east,minimum width=14mm] (s_0) {$s_0$};
			\node[draw,very thick,fit=(r_10) (r_11) (s_1) (s_0) (r_00) (r_11),inner sep=-0mm,label=below:$M^1_2$] {};

			\matrix[box,label=below:$M^0_3$] at (7,2) { 
				\node {};  \\ \node (s_0') {$r_1$};  \\ \node (s_1') {$r_0$}; \\ };			
						
			\node at (-2,0) {$A^1$};
			\node at (-2,2) {$A^0$};
			
			\node at (1,0.4) {$y_0$}
				edge [->] (s_0.west)
				edge [->] (s_2.east);
					
			\node at (4.5,2.5) {$x_1$}
				edge [->] (r_00.north)
				edge [->] (s_0');

			\node at (5.5,1.5) {$x_0$}
				edge [->] (r_01.north)
				edge [->] (s_1');
								
		\end{tikzpicture}
	\end{center}
	\caption{The follower $y_0$ was permitted. Followers $x_0$ and $x_1$ use $y_0$'s boxes. $y_1$ is cancelled.}
	\label{fig4}
\end{figure}

\begin{figure}[h]

	\begin{center}
		\begin{tikzpicture}
			[box/.style = {draw,very thick,minimum size=7mm,nodes={draw,very thin},inner sep=0mm},
			bbox/.style = {draw,very thin,minimum size=7mm,inner sep=0mm,on grid}]

			\matrix[box,label=below:$M^0_2$] (M_02) at (0,2) { \node{}; \\ \node (s_2) {$s_2$};  \\ };
			
			\node[bbox,anchor=south east] (r_10) at (2.5,0) {};
			\node[bbox,right=7mm of r_10] (r_11) {$r$};
			\node[bbox,right=7mm of r_11] (r_00) {};
			\node[bbox,right=7mm of r_00] (r_01) {};
			\node[bbox,below=0mm of r_10.south east,minimum width=14mm] (s_1) {$s_1$};
			\node[bbox,below=0mm of r_00.south east,minimum width=14mm] (s_0) {$s_0$};
			\node[draw,very thick,fit=(r_10) (r_11) (s_1) (s_0) (r_00) (r_11),inner sep=-0mm,label=below:$M^1_2$] {};

			\matrix[box,label=below:$M^0_3$] at (7,2) { 
				\node {};  \\ \node (s_0') {$r$};  \\ \node (s_1') {$s_0$}; \\ };			
						
			\node at (-2,0) {$A^1$};
			\node at (-2,2) {$A^0$};
			
			\node at (1,0.4) {$y_1$}
				edge [->] (s_1.west)
				edge [->] (s_2.east);
					
			\node at (4,2) {$x$}
				edge [->] (r_11.north)
				edge [->] (s_0');
				
			\node at (5.5,0.7) {$y_0$}
				edge [->] (s_0.east)
				edge [->] (s_1');
								
		\end{tikzpicture}
	\end{center}
	\caption{In an alternate reality, $y_1$ was permitted. A new follower $x$ uses $y_1$'s boxes.}
	\label{fig5}
\end{figure}
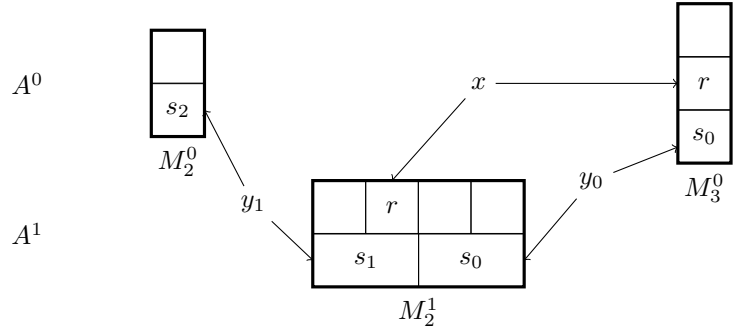
% 
% \begin{figure}[h]
% 	\begin{center}
% 	\begin{tikzpicture}
% 		[basic/.style = {draw,very thin,minimum height=7mm,inner sep=0mm},
% 		 first/.style = {basic,minimum width=7mm},
% 		  second/.style = {basic,minimum width=21mm},
% 			third/.style = {basic,minimum width=63mm}]
% 
% 		\node[first] (111) {};
% 		\node[first,right=0mm of 111] (112) {};
% 		\node[first,right=0mm of 112] (113)	{};
% 
% 		\node[second,below=0mm of 112.south] (11) {};
% 		\node[second,right=0mm of 11] (12) {};
% 		\node[second,right=0mm of 12] (13)	{};
% 
% 		\node[third,below=0mm of 12.south] (1) {};
% 		\node[second,right=0mm of 1] (2) {};
% 		\node[second,right=0mm of 2] (3)	{};
% 
% 		% \node[third] (1) {};
% 		% \node[third,right=0mm of 1] (2) {};
% 		% \node[third,right=0mm of 2] (3) {};
% 		% 
% 		% \node[second,anchor=south west,above=0mm of 1.north west] (11) {};
% 		
% 		
% 		
% 		
% 	\end{tikzpicture}	
% 	\end{center}
% 	\caption{caption}
% 	\label{fig3box}
% \end{figure}

We see the general structure of the metaboxes $M^i_k$ for $i<2$ and $k<\w$. Consider for example $i=1$. A follower $x$, pointing at $M^1_k$, also points at either $M^0_k$ or $M^0_{k+1}$. In the first case, it is asking for permission from $A^1$; in the second, from $A^0$. For such a follower $x$, let $M^1(x)$ be the sub-box of $M^1_k$ which is used by $x$. If $x$ and $y$ point to $M^1_k$ and $M^0_{k+1}$ (we write $0=\top(x)=\top(y)$ to indicate that $x$ and $y$ are waiting for permission from $A^0$), then we require that $M^1(x)$ and $M^1(y)$ are disjoint, so that the $A^1$-use of $x$, in case it never moves again, does  not go to infinity. If $x$ and $y$ point to $M^1_k$ and $M^0_k$ (i.e.\ $1=\top(x)= \top(y)$) and $x$ is stronger than $y$ then we require that $M^1(y)\subseteq M^1(x)$, and that $t^1(x)<t^1(y)$. Here $t^1(x)$ is the stage at which $x$ was last moved, which is the same as the value of $\psi(z)$ which is enumerated into $T^1(z)$ for $z\in M^1_k(x)$, the uses $u^1(z,t^1(x))$ of said enumerations determining the use connecting $x\in E$ to $A^1$. The conditions $M^1(y)\subseteq M^1(x)$ and $t^1(x)<t^1(y)$, together with the bound $k$ on the size of $T^1(z)$ for $z\in M^1_k$, ensure that there are at most $k$ followers $x$ with $\top(x)=1$ pointing at $M^1_k$. This allows us to calculate how large $M^0_k$ needs to be (as we mentioned, about $k^k$), and so define the order function $h$ ahead of time. The metabox $M^1_k$ is structured as a tree of sub-boxes, ordered by reverse inclusion. A stem consists of the boxes $M^1(x)$ for followers $x$ with $\top(x)=1$; each element $M^1(x)$ of the stem is divided into disjoint sub-boxes $M^1(y)$, where for all but at most one $y$, we have $\top(y)=0$. 

For $i=0$ instead of 1, the situation is similar; if $x$ points at $M^0_k$ (we write $k=k^0(x)$), then $\top(x)=1$ if it also points at $M^1_k$, and $\top(x)=0$ if it also points at $M^1_{k-1}$. A typical walk by a follower $x$ starts with $M^1_k$ and $M^0_k$ for some $k$ (so $\top(x)$ begins with 1); an $A^1$-permission moves $x$ to point at $M^0_k$ and $M^1_{k-1}$ (moving $\top(x)$ to 0); then an $A^0$-permission moves $x$ to point at $M^1_{k-1}$ and $M^0_{k-1}$ (moving $\top(x)$ back to 1), and so on, until at some point, one of the boxes it points to is a (possibly promoted) 1-box, which allows dual permission and enumeration into $E$. 

\medskip

This structure explains why we need to order the followers by priority, and cancel weak followers whenever a stronger follower moves. If $x$ is stronger than $y$, then for both $i<2$, the use $v^i(x)$ for reducing $x\in E$ to $A_i$ is smaller than the corresponding use $v^i(y)$. If $x$ receives permission from say  $A^0$ at some stage, then the $A^0$-change below $v^0(y)$ invalidates any promotion of an $A^0$-box credited to $y$. On the other hand, on the $A^1$-side, the boxes $M^1(x)$ and $M^1(y)$ are disjoint; if $y$ were not cancelled, we would have to move $y$ to also require permission from $A^1$, violating the requirement that $M^1(x)$ and $M^1(y)$ are comparable in this case. And again we emphasise, that without this comparability, the whole mechanism, of using the bound on the size of the traces to bound the number of boxes that we need, is thrown out of whack. 

\

To complete the discussion, there are three further issues we need to address. The first is simple: we notice that followers for different requirements need to all work in concert in the grand scheme of movement and promotion. That is, we cannot, for different requirements $P^e$, set up separate boxes, only devoted to followers of $P^e$. Again, this is because for each $k$, we can overall have only finitely many $k$-boxes, as $h$ has to be unbounded. The \emph{cursus honorum} of followers for $P^e$ may start with $e$-boxes, say, but may require advancement all the way down to 1-boxes. The effort is combined; followers for a requirement $P^e$ will make use of promotions credited to followers of stronger requirements. The important thing is that even though some followers get stuck, overall there are no losses but only collective gains. 

\medskip

The discussion above only referred to realised followers. Before a follower $x$ is realised, it needs to point at boxes $M^i_k$ for both $i<2$, but permissions from $A^1$ where say $1 = \top(x)$ are not useful, as we do not want to enumerate $x$ into $E$ before it is realised. So $x$ cannot play a full part in the global promotion process. This is why we need to allocate to it a \emph{private} $A^1$-box, disjoint from the general structure of sub-boxes of $M^1_k$. On the other hand, when $x$ is realised and receives its first $A^1$-permission, it needs to join the global effort, while of course not moving its $A^0$-pointer. This is why the $A^0$-boxes associated with an unrealised follower are not private, but part of the general structure of boxes. 

\medskip

Finally, the impatient reader has likely already been wondering for a while, what if we do not have any 1-boxes? That is, what happens if the constant $c$ supplied by the recursion theorem is greater than 1? This situation seems to invalidate the entire plan, since it loses its basis of 1-boxes: if we have none, then there is no mechanism for eventually enumerating any followers into $E$. 

The solution is to introduce nonuniformity to the process of reducing $E$ to $A^1$. Suppose that a follower $x$ has been promoted as much as it could: it points to the smallest $A^1$-box $M^1_c$, and to the smallest $A^0$-box $M^0_c$, and $\top(x)=1$. It then receives $A^1$-permission, but we cannot guarantee $A^0$-permission, since $c>1$. We then, contrary to everything we said so far, leave the $A^1$-permission open: we do not appoint a new $A^1$-use, and leave $x$ to point only at $M^0_c$, and not at any $A^1$-box. The situation of $x$ in the scheme of sub-boxes of $M^0_c$ is as if it were pointing at $M^1_{c-1}$, if that metabox existed. So the boxes $M^0(x)$ of such followers are nested, and so there are at most $c$ of them. If $x$ is then permitted by $A^0$, we enumerate it into $E$; otherwise we do not. We then see that the promotion structure implies that there can be at most $c$ followers for which the $A^1$-permission is left open but are not later cancelled or enumerated into $E$. Thus $A^1$, upon leaving a permission open for $x$, searches for a later stage at which $x$ is cancelled or enumerated into $E$, and thus can find out if $x$ is in $E$ or not. This search will halt for all but finitely many, indeed at most $c$ many, followers. 

We note that the nonuniformity we introduced is limited to the reduction of $E$ to $A^1$. Namely, an index for $E$ is obtained effectively from indices for $A^0$ and $A^1$ (and for $A^i$-traces for a universal $\Sigma^0_2$ partial function), and also, an index for a reduction of $E$ to $A^0$ is effectively obtained. This ``near uniformity'' of the construction turns out to be important when we prove Theorem \ref{thm_main}, when we try to generalise the construction we discussed to deal with infinitely many oracles. Looking ahead, as an exercise, the reader is invited to try to show that if $A^0$, $A^1$ and $A^2$ are SJT-hard c.e.\ sets, then there is a noncomputable c.e.\ set $E$ reducible to each of $A^0$, $A^1$ and~$A^2$.

\subsection{Construction}

We can now give the details of the construction. For tidiness, we define two partial $\Sigma^0_2$ functions, $\psi^0$ and $\psi^1$, by defining uniformly computable sequences $\seq{\psi^i_s}$ of total functions, and let $\psi^i$ be the partial limit of the sequence $\seq{\psi^i_s}$. For both $i<2$, we let $\psi_0(n)=0$ for all $n$. At stage $s$ we may define $\psi^i_s(z)=s$ for some $z$'s; for all other $z$'s, we let $\psi^i_s(z)= \psi^i_{s-1}(z)$. We let $\psi^i(z) = \lim_s \psi^i_s(z)$ if the limit exists; otherwise $\psi^i(z)\diverge$. 

As argued, for example, in \cite{CholakDowneyGreenberg:SJT1}, the recursion theorem provides us with a constant $c\ge 1$, and for both $i<2$, an $A_i$-c.e.\ traces $T^i$ which traces $\psi^i$. These traces are both bounded by an order function $h$, which we define; we may use the constant $c$ in the definition, but we need to ensure that the definition of $h$ is uniform in $c$, and that $h(0)\ge c$. 

For $k\ge c$, we define $I(k)$, which are finite intervals of natural numbers. These are successive intervals, so to define these intervals, it suffices to determine their size, which we set to $|I(k)| = 1 + (2k+2)^{k+1}$. This, in turn, determines $h$, because we define $h(z)=k$ for all $z\in I(k)$. 

Let $i<2$. We let $\seq{A^i_s}$ be an effective enumeration of $A^i$. For $z<s$, we let $T^i_{s}(z)$ be the collection of numbers enumerated into $T^i(z)$ with oracle $A^{i}_s$. By delaying enumeration of elements into $T^i_s(z)$, we may assume that for all $s$, $|T^i_s(z)| \le h(z)$. 

For $t\in T^i_{s}(z)$, we let $u^i_s(z,t)$ be the $A^i_s$-use of enumerating $t$ into $T^i_{s}(z)$. Similarly, for $t\in T^i(z)$ we let $u^i(z,t)$ be the $A^i$-use of enumerating $t$ into $T^i(z)$.

Now by the fact that $T^i$ traces $\psi$, we may assume that for all $s$, there is some $t>s$ such that for all $z<s$, for both $i<2$, $\psi^i_s(z)\in T^i_{t}(z)$. To ensure this, when we redefine $\psi^i_s(z)$, we search for a stage $t>s$ as is sought after. While we search for $t$, we hold the definition of $\psi^i$, so that if such a stage is not found, we have $\psi^i = \psi^i_s$. But then we have a contradiction to the fact that $T^i$ traces $\psi^i$.

There are two options for handling this fact. We could speed up the enumeration of the sets $A^0$ and $A^1$ and assume that the stage $t$ provided by the claim equals $s$. This necessitates a cascading effect: a typical response for finding $s\in T^i_s(z)$, which likely involves an $A^i$-change, is to redefine $\psi^i_s(z')=s$ for more values $z'$, and so a further speed-up of $A^0$ and $A^1$ so that $s\in T^i_{s}(z')$ for these inputs $z'$, and so on. We can argue that at each stage, this process repeats only finitely many times. 

We do not take this approach. Mostly, this is because we want to present a construction which is close to the full construction proving Theorem \ref{thm_main}. In that construction, we work with all c.e.\ oracles, not only with two c.e.\ oracles $A^0$ and $A^1$ which are guaranteed to be SJT-hard. So in the full construction, we need to guess which traces, if any, trace the functions $\psi^i$ we build; for some oracles no guess will be correct. In other words, in the full construction, the fact above is only guaranteed to hold for oracles which are SJT-hard, and only for correct guesses of their trace. Hence in the full construction, we cannot speed up all sets and get instant gratification. We have to restrict ourselves to stages at which our guesses seem correct. 

We apply a similar, more patient approach in the current construction. Following Nies's terminology, the construction will take place at a computable set of \emph{stages}. Given a \stage $s$, we let the following \stage be the next stage $t>s$ such that for all $z$ which were encountered by stage $s$, $\psi^i_s(z)\in T^i_{t}(z)$. Between stages $s$ and $t$, no change is made to $\psi^i$, or to any other object of the construction. The fact above now says that there are infinitely many \emph{stages}.

\subsubsection*{Followers}

We try to meet the requirements $P^e$ for $e\ge c$. As mentioned above, a requirement $P^e$ appoints \emph{followers}. A follower $x$ for $P^e$ is \emph{realised} at stage $s$ if $\vphi^e_s(x)=0$. The requirement $P^e$ is \emph{satisfied} at stage $s$ if there is some $x\in E_s$ which is realised for $P^e$. If $P^e$ is satisfied at stage $s$, then $P^e$ takes no action at stage $s$. 

With any follower $x$, ``alive'' at the end of a stage $s$ (that means that $x$ was appointed at some stage $s'\le s$ and was not cancelled at any stage $s''\in (s',s]$), we associate some auxiliary objects which aid with the reduction of the question ``$x\in E$?'' to $A^0$ and $A^1$.  
\begin{itemize}
	\item We attach a number $k^0_s(x)\in [c,e]$, and possibly also a number $k^1_s(x)\in [c,e]$. These are the \emph{levels} at which the $x$-pointer points.
	\item We define a number $\top_s(x)\in \{0,1\}$. This is the side from which $x$ next requires permission. If $\top_s(x)=1$ then $k^1_s(x)$ is defined. 
	\item For $i<2$ such that $k^i_s(x)$ is defined, we define a metabox $M^i_s(x)\subset I\left(k^i_s(x)\right)$. 
	\item For such $i$, we also define a value $t^i_s(x)$. This is the value appearing in the boxes which we tie to the use of reducing $x\in E$ to $A^i$.	
\end{itemize}

Let $s$ be a \stage at the end of which a follower $x$ is alive. Let $i<2$ such that $k^i_s(x)$ is defined. We say that an $i$-computation is defined for $x$ at stage $s$ if for all $z\in M^i_s(x)$ we have $t^i_s(x)\in T^i_s(z)$. In this case, we let 
\[ v^i_s(x) = \max \left\{ u^i_s\left(z,t^i_s(x)\right)  \,:\, z\in M^i_s(x) \right\}.\]
We denote the fact that an $i$-computation for $x$ is defined at stage $s$ by writing $v^i_s(x)\converge$. Otherwise, we write $v^i_s(x)\diverge$. 

Let $s>0$ be a \stage at the beginning of which a follower $x$ is alive. Let $r$ be the last \stage before $s$ (so $x$ was alive by the end of stage $r$). Let $i = \top_r(x)$. We say that $x$ is \emph{permitted} at stage $s$ if $v^i_r(x)\diverge$, or if $A^i_s\rest {v^i_r(x)} \ne A^i_r\rest {v^i_r(x)}$. 

All followers alive at stage $s$ are linearly ordered by priority, which is determined by the stage of their appointment (at most one follower is appointed at each \emph{stage}). As usual, when a new follower is appointed, it is chosen to be large relative to any number previously encountered in the construction, and so if $x$ and $y$ are followers at a stage $s$, then $x$ is stronger than $y$ if and only if $x<y$. 

We say that a follower $x$ \emph{requires attention} at \stage $s$ if it is appointed at stage $s$, or it is permitted at stage $s$. If some follower $x$ requires attention at stage $s$, then the strongest such follower will \emph{receive attention}. If a follower $x$, alive both at the beginning and at the end of a stage $s$, does not receive attention at stage $s$, then there is no change to $x$'s parameters.

If a follower $x$ receives attention at stage $s$, then all followers weaker than $x$ are cancelled. In addition, when a requirement $P^e$ appoints a new follower, all followers for all weaker requirements are cancelled. As a result, if $x$ is a follower for a requirement $P^e$, and $y$ is a follower for a weaker requirement $P^{e'}$, both alive at some stage, then $x$ is stronger than $y$.

\subsubsection*{Carving the boxes}

Let $a(k) = 2k+2$. We define a tree of sub-boxes of $I(k)$, indexed by strings in ${}^{\le k+1}a(k)$. Let $J(k) = \left\{ \min I(k) \right\}$; let $B(k,\seq{}) = I(k)\setminus J(k)$, where $\seq{}$ is the empty string. Note that $|B(k,\seq{})| = a(k)^{k+1}$. Given $B(k,\alpha)\subset I(k)$ for $\alpha\in {}^{\le k}a(k)$, which by induction has size $a(k)^{k+1-|\alpha|}$, we let
\[ \left\{ B(k,\conc{\alpha}{m}) \,:\, m<a(k) \right\} \]
be a partition of $B(k,\alpha)$ into $a(k)$ many subsets of equal size $a(k)^{k-|\alpha|}$. Thus we get a tree of boxes: for $\alpha,\beta \in {}^{\le k+1}a(k)$, if $\alpha\subseteq \beta$ then $B(k,\alpha)\supseteq B(k,\beta)$, and if $\alpha\perp \beta$ then $B(k,\alpha)\cap B(k,\beta) = \emptyset$. Also, for all $\alpha\in {}^{\le k+1}a(k)$, $B(k,\alpha)$ is disjoint from $J(k)$. 

Let $x$ be any follower at the end of some stage $s$. Let $i<2$ such that $k^i_s(x)$ is defined. Then there are two possibilities:

\begin{enumerate}
	\item $M^i_s(x) = J\left(k^i_s(x)\right)$; we sometimes say that $x$ \emph{resides at its private box} at stage $s$; or
	\item $M^i_s(x) = B\left(k^i_s(x),\alpha\right)$ for some $\alpha\in {}^{\le k+1}a(k)$; we let $\alpha^i_s(x)$ be this string $\alpha$.
\end{enumerate}
If $M^i_s(x) = J\left(k^i_s(x)\right)$ ($x$ resides at its private box at stage $s$) then $i=\top_s(x)=1$ and $x$ did not move since it was appointed. 

\medskip

Let $s$ be a \stage; let $i<2$ and $k\ge c$. After it has been decided if any follower receives attention at \stage $s$, and thus cancels weaker followers, the collection of followers which are alive at the end of the stage is determined. Based on this information, we define a string $\beta^i_s(k)$:

\begin{itemize}
	\item Suppose that there is some follower $y$, alive at both the beginning and the end of stage $s$, which does not receive attention at stage $s$, and such that $\top_{s-1}(y)=i$, $k^i_{s-1}(y)=k$, and $M^i_{s-1}(y)\ne J(k)$. We then let $y$ be the weakest such follower, and let $\beta^i_s(k) = \alpha^i_{s-1}(y)$.
	\item If there is no such follower $y$, then we let $\beta^i_s(k) = \seq{}$.
\end{itemize}

The main part of the proof will be the establishment of the following ``book-keeping'' lemma:

\begin{lemma}\label{lem_bookkeeping_minimal_pair}
	For all $s$, $i<2$ and $k\ge c$, $|\beta^i_s(k)| \le k$. Furthermore, there is some $m<a(k)$ such that 
	for all followers $y$ which are alive at both the beginning and the end of stage $s$, if $k^i_{s-1}(y)$ is defined and equals $k$, and $M^i_{s-1}(y)\ne J(k)$, then $\alpha^i_{s-1}(x)\ne \conc{\beta^i_s(k)}{m}$. We let $m^i_s(k)$ be the least such $m$. 
\end{lemma}

Armed with Lemma \ref{lem_bookkeeping_minimal_pair}, we can now give the full instructions for the construction.

\subsubsection*{Construction}

$s=0$ is a \emph{stage}; at stage 0, we do nothing except for defining $\psi^i_0(z)=0$ for all $z<\w$ and both $i<2$. Let $s>0$; let $r$ be the last \stage prior to~$s$. As described above, $s$ is a \stage if for all $z$ mentioned by stage $r$, for both $i<2$, $\psi^i_r(z)\in T^i_s(z)$. If $s$ is not a \emph{stage}, then we do nothing at stage $s$, and we let all objects of the construction maintain their previous values; in particular, for all $z$ and $i<2$, $\psi^i_s(z)=\psi^i_{s-1}(z)= \psi^i_r(z)$. 
 
Suppose that $s$ is a \emph{stage}. A requirement $P^e$ \emph{requires attention} at stage $s$ if it is not yet satisfied, and either:
\begin{enumerate}
	\item  every follower of $P^e$ which is currently alive is realised (this includes the case that it has no followers); or
	\item some realised follower of $P^e$ is permitted. 
\end{enumerate}

If no requirement requires attention, we do nothing, and let all objects of the construction maintain their previous values. Otherwise, we let $P^e$ be the strongest requirement which requires attention at stage $s$. 

\medskip

In the first case, we appoint a new, large follower $x$ for $P^e$. We cancel all followers for requirements weaker that $P^e$. We set up $x$'s parameters as follows:
\begin{itemize}
	\item We define $k^0_s(x) = k^1_s(x) = e$.
	\item We let $\top_s(x)=1$. 
	\item We let $M^1_s(x) = J(e)$ (so $x$ begins its life by residing in its private box). 
	\item By Lemma \ref{lem_bookkeeping_minimal_pair}, we let $M^0_s(x) = B\left(e,\conc{\beta^0_s(e)}{m^0_s(e)}\right)$.
	\item We let $t^0_s(x) = t^1_s(x) = s$. 
\end{itemize}
To facilitate this, we define, for both $i<2$ and all $z\in M^i_s(x)$, $\psi^i_s(z)=s$. For all other $z$, we let $\psi^i_s(z) = \psi^i_{s-1}(z) = \psi^i_r(z)$, and end the stage.  

\medskip

In the second case, let $x$ be the strongest follower for $P^e$ which is permitted at stage $s$. We cancel all followers weaker than $x$. Now we need to promote $x$; there are three cases. 

\medskip

\noindent {\textbf{1.}} If $k^1_{s-1}(x) = k^1_r(x)$ is undefined, this means that $x$ has open permission from $A^1$, and has just received permission from $A^0$ (so $\top_{s-1}(x)=0$ and $k^0_{s-1}(x)=c$). So we now have double permission and so we enumerate $x$ into $E$. As $P^e$ now becomes satisfied, we cancel all the followers for $P^e$. 

\medskip

\noindent {\textbf{2.}} If $\top_{s-1}(x)=1$ and $k^1_{s-1}(x)=c$, then $A^1$ now permits $x$, but there is no stronger $A^1$-box for $x$ to be promoted to. In this case, we leave the $A^1$-permission open from now on. This means that $k^1_s(x)$, and so $M^1_s(x)$ and $t^1_s(x)$, are all undefined. We define $\top_s(x) = 0$, which means that from now we are seeking the $A^0$-permission which will land us in case (1). We leave $k^0_s(x)= k^0_{s-1}(x)$, $M^0_s(x) = M^0_{s-1}(x)$ and $t^0_s(x) = t^0_{s-1}(x)$. 

\medskip

\noindent {\textbf{3.}} If neither case (1) nor case (2) hold then we can have a ``regular'' promotion for $x$. Let $i = \top_{s-1}(x)$. 
\begin{itemize}
	\item We let $\top_s(x) = 1-i$. We leave $k^{1-i}_s(x) = k^{1-i}_s(x)$, $M^{1-i}_s(x) = M^{1-i}_{s-1}(x)$, and $t^{1-i}_s(x) = t^{1-i}_{s-1}(x)$. 
	\item We let $k^i_s(x) = k^i_{s-1}(x)-1$. We note that $k^i_s(x)\ge c$, for otherwise we would be in case (1) (if $i=0$) or case (2) (if $i=1$). 
	\item By Lemma \ref{lem_bookkeeping_minimal_pair}, we let $M^i_s(x) = B\left(k^i_s(x),\conc{\beta^i_s\left(k^i_s(x)\right)}{m^i_s\left(k^i_s(x)\right)}\right)$.
	\item For all $z\in M^i_s(x)$, we let $t^i_s(x) = s$. 	
\end{itemize}
To facilitate this, we set $\psi^i_s(z) = s$ for all $z\in M^i_s(x)$. For all other $z$, we set $\psi^i_s(z) = \psi^i_{s-1}(z)$. For all $z$, we set $\psi^{1-i}_s(z) = \psi^{1-i}_{s-1}(z)$. 

\medskip

In any of the cases above, we then end the stage. This completes the construction. 

\

\subsection{Justification}

Before we verify that all requirements are met, we need to show that the construction can actually be carried out as described: we need to prove Lemma \ref{lem_bookkeeping_minimal_pair}. The proof of this lemma will follow a careful analysis of how metaboxes are used, allowing us to establish bounds on the number of followers processed by these boxes at any given stage. 

\

We first establish some basic facts and notation. To begin, for $e\ge c$ and $s<\w$, let $F^e_s$ be the collection of followers for $P^e$ which are alive at the end of stage $s$. We let $F_s = \bigcup_{e\ge c}F^e_s$ be the collection of all followers alive at the end of stage $s$. %We let $\bar{F_s} = F_{s-1}\cap F_s$. 

For $x\in F_s$, let $R_s(x) = \{0,1\}$ if both $k^0_s(x)$ and $k^1_s(x)$ are defined, and let $R_s(x) = \{0\}$ if $k^1_s(x)$ is undefined. %For $i\in R_s(x)$, we let $\Par^i_s(x) = \seq{k^i_s(x), M^i_s(x), t^i_s(x)}$. 

The following is immediate from the construction.

\begin{lemma}\label{lem_MP_k_is_monotone}
	Let $x\in F_s$. 
	\begin{enumerate}
		\item For $i\in R_s(x)$, we have $k^i_s(x)\in [c,e]$. 
		\item $\top_s(x)\in R_s(x)$. 
		\item Exactly one of the following holds:
		\begin{itemize}
			\item $k^0_s(x) = k^1_s(x)$ and $\top_s(x)=1$;
			\item $k^0_s(x) = k^1_s(x)+1$ and $\top_s(x) = 0$;
			\item $R_s(x) = \{0\}$ and $k^0_s(x) = c$.
		\end{itemize}
		\item If $t<s$, $x\in F_t$ and $i\in R_s(x)$, then $i\in R_t(x)$ and $k^i_s(x)\le k^i_t(x)$. 
		\item Suppose that $x\in F_{s-1}\cap {F_s}$ and $i\in R_s(x)$. If $k^i_{s-1}(x)\ne k^i_s(x)$ then $i=\top_{s-1}(x)$, $1-i=\top_s(x)$, and $k^i_s(x)  = k^i_{s-1}(x)-1$. If $r<s$, $x\in F_r$ and $k^i_r(x) = k^i_{s}(x)$ then $M^i_r(x) = M^i_s(x)$ and $t^i_r(x) = t^i_s(x)$. 
	\end{enumerate}	
\end{lemma}

In light of part (5) of Lemma \ref{lem_MP_k_is_monotone}, for any follower $x$, $k\ge c$ and $i<2$, the value of $\Par^i_s(x)$ for stages $s$ such that $x\in F_s$, $i\in R_s(x)$ and $k^i_s(x)=k$ does not depend on the choice of such a stage $s$. If there is any such stage, we let $\Par^i(k,x) = \seq{k,M^i(k,x),t^i(k,x)}$ be this value. If $M^i(k,x)\ne J(k)$, then we let $\alpha^i(k,x) = \alpha^i_s(x)$ for such a stage $s$ be the string $\alpha$ such that $M^i_k(x) = B(k,\alpha)$. 

\

We give names to sets of followers, in light of part (3) of Lemma \ref{lem_MP_k_is_monotone}. Let $i<2$, and let $k\ge c$. Let $s<\w$.
\begin{itemize}
	\item We let $K^i_s(k)$ be the collection of followers $x\in F_s$ such that $i=\top_s(x)$, $k = k^i_s(x)$, and $M^i_s(x) = M^i(k,x)$ is not $J(k)$; that is, $x$ does not reside at its private box at the end of stage $s$.
	\item We let $L^i_s(k)$ be the collection of followers $x\in F_s$ such that $i=\top_s(x)$, $k=k^i_s(x)$, and $M^i_s(x) = M^i(k,x) = J(k)$. Indeed $L^i_s(k)$ is nonempty only if $i=1$, and every $x\in L^1_s(k)$ is a follower for $P^k$. 
	\item We let $G^i_s(k)$ be the collection of followers $x\in F_s$ such that $i\ne \top_s(x)$ but $i\in R_s(x)$ and $k^i_s(x) = k$. 
\end{itemize}
For any $i$, $k$ and $s$, the sets $K^i_s(k)$, $L^i_s(k)$ and $G^i_s(k)$ are pairwise disjoint. 

For brevity, we let $KG^i_s(k) = K^i_s(k)\cup G^i_s(k)$, $LG^i_s(k) = L^i_s(k)\cup G^i_s(k)$, and so on. So:
\begin{itemize}
	\item $KLG^i_s(k)$ is the collection of followers $x\in F_s$ such that $i\in R_s(x)$ and $k^i_s(x) = k$;
	\item $KL^i_s(k)$ is the collection of followers $x\in KLG^i_s(k)$  such that $i = \top_s(x)$; and
	\item $KG^i_s(k)$ is the collection of followers $x\in KLG^i_s(k)$ such that $\alpha^i_s(x)$ is defined, i.e.\ such that $M^i_s(x)\ne J(k)$. 
\end{itemize}

For $X\in \{K,L,G,KL,KG,LG,KLG \}$, we let $\bar{X^i_s}(k) = X^i_s(k)\cap X^i_{s-1}(k)$. 

\

The following lemma translates the construction into our new notation. In all of the following lemmas, let $i<2$, $k\ge c$ and $s<\w$ be a \emph{stage}.

\begin{lemma}\label{lem_MP_KLG_and_t}
	Let $x\in KLG^i_s(k)$. Let $t=t^i_s(x) = t^i(k,x)$. The \stage $t$ is is the least stage $r$ such that $x\in KLG^i_r(k)$. At stage $t$, $x$ is placed into $LG^i(k)$. 
	\begin{itemize}
	\item If $x$ was appointed at stage $t$, then $x$ is placed into $L^1_t(k)$ and $G^0_t(k)$.
	\item Otherwise, $x$ is realised at stage $t$, $i=\top_{t-1}(x)$ (and $1-i=\top_t(x)$), $x$ is extracted from $KL^i_{t-1}(k+1)$ and placed into $G^i_t(k)$; and moved from $G^{1-i}_{t-1}(k')$ to $K^{1-i}_{t}(k')$, where $k' = k^{1-i}_t(x)$. 
	\end{itemize}
	
	As a corollary, if $x\in KG^i_s(k)$, then $x$ enters $G^i(k)$ at stage $t$, and so $M^i(k,x) = B\left(k,\conc{\beta^i_t(k)}{m^i_t(k)}\right)$. The string $\beta^i_t(k)$ is defined as follows:
	\begin{itemize}
		\item If $\bar{K^i_t}(k) = \emptyset$ then $\beta^i_t(k) = \seq{}$. 
		\item If $\bar{K^i_t}(k)\ne \emptyset$ then $\beta^i_t(k) = \alpha^i(k,w)$, where $w = \max \bar{K^i_t}(k)$. 
	\end{itemize}

	The number ${m^i_t(k)}$ is chosen so that for all $y\in \bar{KG^i_t}(k)$, $\alpha^i(k,y)\ne \conc{\beta^i_t(k)}{m^i_t(k)}$.
\end{lemma}

Markers obey the priority ordering:

\begin{lemma}\label{lem_MP_marker_size_obeys_priority}
	Let $x,y\in F_s$ with $x<y$. Let $i\in R_s(x)$ and $j\in R_s(y)$. Then $t^i_s(x)< t^j_s(y)$. 
\end{lemma}

\begin{proof}
	Let $r$ be the stage at which $y$ is appointed; so $t^j_s(y)\ge r$. At stage $t^i_s(x)$, $x$ receives attention; if $t^i_s(x)\ge r$ then $y$ would be cancelled at stage $t^i_s(x)$. Hence $t^i_s(x)<r$. 
\end{proof}

\subsubsection*{The tree of occupied boxes}

Let 
\[ O^i_s(k) = \left\{ \alpha^i(k,x)  \,:\, x\in KG^i_s(k) \right\}.\]
This is the collection of indices of metaboxes which are occupied at stage $s$. 

\begin{lemma}\label{lem_alpha_i_k_x_is_injective}
	The function $x\mapsto \alpha^i(k,x)$ as $x$ ranges over the followers in $KG^i_s(k)$ is injective. 
\end{lemma}

\begin{proof}
	Let $y<x$ be elements of $KG^i_s(k)$. Let $t = t^i(k,x)$. As in the proof of Lemma \ref{lem_MP_marker_size_obeys_priority}, $y$ does not receive attention between stages $t$ and $s$, as this would cancel $x$. Hence $y\in \bar{KG^i_t}(k)$. Now by Lemma \ref{lem_MP_KLG_and_t} we see that the choice of $\alpha^i(k,x)$ at stage $t$ ensures that $\alpha^i(k,y)\ne \alpha^i(k,x)$. 
\end{proof}

We note that $O^i_s(k)$ does not contain the empty string. This is because every $\alpha^i(k,x)$ is chosen to be $\conc{\beta^i_t(k)}{m^i_t(k)}$ for $t = t^i(k,x)$. 

\begin{lemma}\label{lem_MP_occupied_boxes_are_a_forest} 
 Let $x\in KG^i_s(k)$. If $\gamma\subsetneq \alpha^i(k,x)$ is nonempty, then there is some $y\in K^i_s(k)$, stronger than $x$, such that $\gamma = \alpha^i(k,y)$. 
\end{lemma}

\begin{proof}
	Let $x\in KG^i_s(k)$, and suppose that $|\alpha^i(k,x)|>1$; let $\gamma = \alpha^i_k(x)^-$ be the immediate predecessor of $\alpha^i(k,x)$. We show that there is some $y<x$ in $K^i_s(k)$ such that $\gamma = \alpha^i(k,y)$. Then part (2) follows by induction on $|\alpha^i(k,x)|$. 
	
	Let $t = t^i(k,x)$. At stage $t$ we choose $\alpha^i(k,x) = \conc{\beta^i_t(k)}{m^i_t(k)}$, so $\gamma = \beta^i_t(k)$. Since $\beta^i_t(k) = \gamma \ne \seq{}$, we have $\bar{K^i_t}(k)\ne \emptyset$. Hence (Lemma \ref{lem_MP_KLG_and_t}) $\gamma = \alpha^i(k,w)$ for $w = \max \bar{K^i_t}(k)$. Since $w$ was not cancelled at stage $t$, $w$ is stronger than $x$. Since $x$ is not cancelled between stages $t$ and $s$, $w$ does not receive attention between these stages, and so $w\in K^i_s(k)$. 	
\end{proof}

Lemma \ref{lem_MP_occupied_boxes_are_a_forest} implies that ordered by inclusion, $O^i_s(k)$ is a forest, and that for all $x\in G^i_s(k)$, $\alpha^i(k,x)$ is a leaf of $O^i_s(k)$. It follows that for all $x\in KG^i_s(k)$, at stage $t = t^i(k,x)$, $\alpha^i(k,x)$ is a leaf of $O^i_t(k)$, as $x\in G^i_t(k)$.

The leaves of $O^i_s(k)$ all issue from a single stem $\left\{\alpha^i(k,x)  \,:\, x\in K^i_s(k) \right\}$:

\begin{lemma}\label{lem_MP_occupied_K_boxes_are_comparable}
	Let $x,y\in K^i_s(k)$ with $y<x$. Then $\alpha^i(k,y)\subsetneq \alpha^i(k,x)$. 
\end{lemma}

And so $M^i(k,x)\subsetneq M^i(k,y)$. 

\begin{proof}
	The lemma is proved by induction on $s$. Let $t = t^i(k,x)$. Since $y$ is stronger than $x$ and $y\in K^i_s(k)$, we have $y\in \bar{K^i_t}(k)$. Hence $\bar {K^i_t}(k)\ne \emptyset$, and so $\beta^i_t(k) = \alpha^i(k,w)$ for $w = \max \bar{K^i_t}(k)$. By its definition, $y\le w$. By induction, as $t-1 < t \le s$, we have $\alpha^i(k,y)\subseteq \alpha^i(k,w)$. By its choice, we have $\alpha^i(k,x)\supset \beta^i_t(k) = \alpha^i(k,w)$. Hence $\alpha^i(k,y)\subsetneq \alpha^i(k,x)$. 
\end{proof}

Let 
\[ \bar{O^i_s}(k) = \left\{ \alpha^i(k,x)  \,:\, x\in \bar{KG^i_s}(k) \right\}.\]
We can rephrase our main goal, Lemma \ref{lem_bookkeeping_minimal_pair}, using this notation. The lemma says that $|\beta^i_s(k)|\le k$ and that there is some $m<a(k)$ such that $\conc{\beta^i_s(k)}{m}\notin \bar {O^i_s}(k)$.

\subsubsection*{Impermissiveness}

We work toward finding bounds on the sizes of the sets of followers we have defined. To do so, we tie followers to values in traces. The first lemma shows that if the value $t^i(k,x)$ of a follower is not traced, then the follower $x$ is permitted. 

\begin{lemma}\label{lem_MP_if_untraced_then_permitted}
	Let $s$ be a \emph{stage}. Let $x\in F_{s-1}$, let $i = \top_{s-1}(x)$, and suppose that there is some $z\in M^i_{s-1}(x)$ such that $t^i_{s-1}(x)\notin T^i_s(z)$. Then $x$ is permitted at stage~$s$. 
\end{lemma}

\begin{proof}
	Suppose that $x$ is not permitted at stage $s$. Let $r$ be the previous \stage before stage $s$; so $x\in F_r$, $i=\top_r(x)$ and $t = t^i_r(x)$. Since $x$ is not permitted at stage $s$, an $i$-computation is defined for $x$ at stage $r$, and $A^i_s\rest{v^i_r(x)} = A^i_s\rest{v^i_r(x)}$. Let $z\in M^i_{s-1}(x) = M^i_r(x)$. Then $t = t^i_r(x)\in T^i_r(x)$, and by its definition, $v^i_r(x)\ge u^i_r(z,t)$. The fact that $A^i$ did not change below $v^i_r(x)$ between stages $r$ and $s$ shows that $t\in T^i_s(x)$ as well. 
\end{proof}

The second lemma says that we do not change $\psi^i$ too often. This will also be useful in the verification, to show that the use of reducing $E$ to $A^i$ does not go to infinity. 
 
\begin{lemma}\label{lem_MP_no_changes_to_psi}
	Suppose that either
	\begin{itemize}
		\item $x\in G^i_s(k)$, or
		\item $x\in L^i_s(k)$, and $x$ is unrealised at stage $s$. 
	\end{itemize}
	Then for all $z\in M^i(k,x)$, $\psi^i_s(z) = t^i(k,x)$. 
\end{lemma}

\begin{proof}
	Let $z\in M^i(k,x)$ (in the second case, $z$ is the unique element of $M^i(k,x)$). At stage $t = t^i(k,x)$, we set $\psi^i_t(z)=t$. We need to show that at no stage $r\in (t,s]$ do we redefine the value of $\psi^i(z)$ to be $r$. 
	
	In the second case, this follows from the fact that as $z = \min I(k)$ is the unique element of $J(k)$, we only define $\psi^i_r(z)=r$ at a stage $r$ if at that stage, the requirement $P^k$ appoints new follower. However, the assumption that $x$ is unrealised at stage $s$ implies that $x$ is unrealised at all stages $r\in [t,s]$, and the instructions tell $P^k$ to appoint a new follower only when all of its followers are realised. So at no stage $t\in (t,s]$ does $P^k$ appoint a new follower, and so $\psi^i(z)$ is unchanged between stages $t$ and $s$. 
	
	In the first case, suppose, for contradiction, that at stage $r\in (t,s]$ we redefine $\psi^i_r(z)=r$. Since $z\in I(k)\setminus J(k)$, by the instructions, there is some follower $y$ which enters $G^i(k)$ at stage $r$, and $z\in M^i(k,y)$. By Lemma \ref{lem_MP_occupied_boxes_are_a_forest}, $\alpha^i(k,y)$ is a leaf of $O^i_r(k)$. Since $x\in G^i_s(k)$, we must have $x\in G^i_r(k)$, and so by the same lemma, $\alpha^i(k,x)$ is also a leaf of $O^i_r(k)$. Since $x$ does not enter $G^i(k)$ at stage $r$ (as $r>t$), we have $x\ne y$. Hence $\alpha^i(k,x)$ and $\alpha^i(k,y)$ are incomparable, which means that $M^i(k,x)$ and $M^i(k,y)$ are disjoint, contradicting $z\in M^i(k,x)\cap M^i(k,y)$. 
\end{proof}

\begin{lemma}\label{lem_MP_followers_in_KLG_are_traced}
	Let $s$ be a \emph{stage}. For all $x\in \bar{KLG^i_s}(k)$, for all $z\in M^i(k,x)$, $t^i(k,x)\in T^i_s(z)$. 
\end{lemma}

So $v^i_s(x)\converge$. 

\begin{proof}
	There are three possibilities: either $x\in K^i_{s-1}(k)$, or $x\in G^i_{s-1}(k)$, or $x\in L^i_{s-1}(k)$. 
	
	\medskip
	
	If $x\in K^i_{s-1}(k)$ (so $i = \top_{s-1}(x)$), then the fact that $x\in K^i_s(k)$ means that $x$ is not permitted at stage $s$. For otherwise, either $x$ is cancelled at stage $s$ by some stronger follower, of $x$ receives attention at stage $s$, in which case $x$ is extracted at stage $s$ from $K^i(k)$ and possibly moved to $K^i_s(k-1)$. In this case, the conclusion follows from Lemma \ref{lem_MP_if_untraced_then_permitted}.
	
	\medskip
	
	Now suppose that $x\in G^i_{s-1}(k)$. Let $z\in M^i(k,x)$. By Lemma \ref{lem_MP_no_changes_to_psi}, $\psi^i_{s-1}(z) = t^i(k,x)$. The conclusion then follows from the fact that $s$ is a \emph{stage}: let $r$ be the previous \stage prior to $s$. As $z\in M^i_{s-1}(x) = M^i_r(x)$, $z$ is mentioned by stage $r$, and so $\psi^i_{s-1}(z)=\psi^i_r(z)\in T^i_s(z)$. 
 	
	\medskip
	
	Finally, suppose that $x\in L^i_s(k)$ (so $i=1$). There are two sub-cases. If $x$ is realised at stage $s$, then by the fact that $x\in L^i_s(k)$ it follows that $x$ is not permitted at stage $s$. Then the conclusion follows from Lemma \ref{lem_MP_if_untraced_then_permitted}. 
	
	We assume, then, that $x$ is unrealised at stage $s$. Let $z$ be the unique element of $M^i(k,x) = J(k)$. By Lemma \ref{lem_MP_no_changes_to_psi}, $\psi^i_{s-1}(z) = t^i(k,x)$. As in the second case above, the conclusion follows from the fact that $s$ is a \emph{stage}. 
\end{proof}

\subsubsection*{Sizes of sets of followers}

We can now bound the sets of followers.

\begin{lemma}\label{lem_MP_bound_on_K}
	For all $s$, $\left|K^i_s(k)\right|\le k$.
\end{lemma}

\begin{proof}
	Since $K^i(k)$ does not change between \emph{stages}, we may assume that $s$ is a \emph{stage}. Assuming that $K^i_s(k)$ is nonempty, let $w = \max K^i_s(k)$. By Lemma \ref{lem_MP_occupied_K_boxes_are_comparable}, for all $x\in K^i_s(k)$, $M^i(k,w)\subseteq M^i(k,x)$. Let $z$ be any element of $M^i(k,w)$, and let $x\in K^i_s(k)$. Then $x\in \bar{GK^i_s}(k)$ (Lemma \ref{lem_MP_KLG_and_t}). By Lemma \ref{lem_MP_followers_in_KLG_are_traced}, $t^i(k,x)\in T^i_s(z)$. By Lemma \ref{lem_MP_marker_size_obeys_priority}, the map $x\mapsto t^i(k,x)$ is injective on $K^i_s(k)$. The conclusion follows from the fact that $|T^i_s(z)|\le k$ as $z\in I(k)$ (so $h(z)=k$).
\end{proof}

Lemma \ref{lem_MP_bound_on_K} implies the first part of Lemma \ref{lem_bookkeeping_minimal_pair}: for all $s$, $|\beta^i_s(k)|\le k$. For if $\beta^i_s(k)\ne \seq{}$, then every nonempty initial segment of $\beta^i_s(k)$ equals $\alpha^i(k,x)$ for some $x\in K^i_s(k)$; this follows from the fact that $\beta^i_s(k) = \alpha^i\left(k,\max K^i_s(k)\right)$ and from Lemma \ref{lem_MP_occupied_boxes_are_a_forest}.

\begin{lemma}\label{lem_MP_bound_on_L}
	For all $s$, $\left| L^i_s(k) \right| \le k+1$.
\end{lemma}

Of course $L^0_s(k)$ is always empty, so this is of interest for $i=1$. 

\begin{proof}
	Again we may assume that $s$ is a \emph{stage}. Let $z = \min I(k)$ be the unique element of $J(k)$; so $h(z)=k$, which means that $|T^i_s(z)|\le k$. By Lemma \ref{lem_MP_followers_in_KLG_are_traced}, for all $x\in \bar{L^i_s(k)}$, $t^i(k,x)\in T^i_s(z)$; with Lemma \ref{lem_MP_marker_size_obeys_priority} we see that $\left| \bar{L^i_s}(k) \right| \le k$. At stage $s$, at most one follower is added to $L^i(k)$. 
\end{proof}

So we see that $|KL^i_s(k)| \le 2k+1$. Recall that $a(k) = 2k+2$. 

\begin{lemma}\label{lem_MP_bound_on_G}
	For all $s$, $\left| G^i_s(k) \right| < a(k)$. 
\end{lemma}

\begin{proof}
	Let $x\in G^i_s(k)$. By Lemma \ref{lem_MP_KLG_and_t}, if $i=0$ then $x\in LK^1_s(k)$, and if $i=1$, then $x\in K^0_s(k+1)$. Hence $|G^0_s(k)| \le 2k+1 < a(k)$ and $|G^1_s(k)| \le k+1 < a(k)$. 
\end{proof}

We can finally prove Lemma \ref{lem_bookkeeping_minimal_pair}. 

\begin{proof}[Proof of Lemma \ref{lem_bookkeeping_minimal_pair}]
	We have already seen that Lemma \ref{lem_MP_bound_on_K} implies the first part of the lemma: that $|\beta^i_s(k)|\le k$. It remains to show that there is some $m<a(k)$ such that $\conc{\beta^i_s(k)}{m}\notin \bar{O^i_s}(k)$. We observe that for all $m<a(k)$, if $\conc{\beta^i_s(k)}{m}\in \bar{O^i_s}(k)$, then $\conc{\beta^i_s(k)}{m} = \alpha^i(k,x)$ for some $x\in G^i_{s-1}(k)$. This would imply the lemma, using the bound on the size of $G^i_{s-1}(k)$ given by Lemma \ref{lem_MP_bound_on_G}. To see that $\conc{\beta^i_s(k)}{m}$ cannot be $\alpha^i(k,x)$ for $x\in K^i_{s-1}(k)$, suppose for contradiction that it is; then $x\in K^i_s(k)$ (as it is in $KG^i_s(k)$), and so $\bar{K^i_s}(k)$  is nonempty, whence $\beta^i_k(s) = \alpha^i(k,w)$ for $w = \max \bar{K^i_s}(k)$. But then $x\le w$ and so (Lemma \ref{lem_MP_occupied_K_boxes_are_comparable}) we would have $\alpha^i(k,x)\subseteq \alpha^i(k,w)$ for a contradiction. 
\end{proof}

\subsection{Verification}

Having shown that the construction can proceed as described, we now show that it succeeds in enumerating a set with the desired properties. 

\begin{lemma}\label{lem_MP_infinitely_many_stages}
	There are infinitely many \emph{stages}. 
\end{lemma}

\begin{proof}
	Indicated above. Let $s$ be a \emph{stage}, and suppose for a contradiction that there is no greater \emph{stage}. Then for both $i<2$, $\psi^i = \psi^i_s$. For every input $z$ mentioned by stage $s$, $\psi_i(z)\in T^i(z)$. There is a sufficiently large stage $t>s$ such that for all such $z$ (as there are only finitely many of them), $\psi_i(z)\in T^i_s(z)$ by an $A^i$-correct computation. Hence there must be a \emph{stage} greater than $s$ after all, contradiction. 
\end{proof}

\subsubsection*{Fairness and diagonalisation}

\begin{lemma}\label{lem_MP_followers_move_finitely_much}
	Every follower receives attention only finitely many times. 
\end{lemma}

\begin{proof}
	For $x\in F_s$, let $k_s(x) = \sum_{i\in R_s(x)} k^i_s(x)$. If $x\in F_{s-1}\cap F_s$ then $k_s(x)\le k_{s-1}(x)$ (Lemma \ref{lem_MP_KLG_and_t}). 

	Suppose that $x\in F_{s-1}\cap F_s$ is a follower which receives attention at stage $s$. Then either $k^0_s(x)< k^0_{s-1}(x)$ or $k^1_s(x) < k^1_{s-1}(x)$. Hence $k_s(x)< k_{s-1}(x)$. This can happen at most finitely many times. Indeed, after being appointed, each follower for requirement $P^e$ can receive attention at most $2e$ many times. 
\end{proof}

For $i<2$, $k\ge c$ and $X\in \{K,L,G\}$, let $X^i_\w(k) = \lim_s X^i_s(k)$ be the collection of followers $x$ which are in $X^i_s(k)$ for all but finitely many $s$. Lemma \ref{lem_MP_bound_on_K} shows that $|K^i_\w(k)|\le k$ and Lemma \ref{lem_MP_bound_on_L} shows that $|L^i_\w(k)|\le k+1$ (in fact, the proof of Lemma \ref{lem_MP_bound_on_L} shows that $|L^i_\w(k)|\le k$; but from now, we only care that it is finite). 

\begin{lemma}\label{lem_MP_reqs_are_met_and_fair}
	For every $e\ge c$, the requirement $P^e$ is met, and there is some stage after which no follower for $P^e$ ever requires attention. In particular, eventually $P^e$ stops enumerating new followers. 
\end{lemma}

\begin{proof}
	By induction on $e\ge c$. Suppose this has been verified for all $e'<e$.
	
	Of course, if $P^e$ is ever satisfied, then it is met, and ceases all action. We assume, then, that no follower for $P^e$ is ever enumerated into $E$. 
	
	 Let $H^e$ be the collection of followers $x$ for $P^e$ which are never cancelled. Then by Lemma \ref{lem_MP_followers_move_finitely_much}, 
	\[ H^e\subseteq L^1_\w(e) \cup \bigcup_{k\in [c,e], i<2} K^i_\w(k) .\]
	This shows that $H^e$ is finite. 
	
	Let $s_0$ be the last stage at which any follower for a requirement stronger than $P^e$ receives attention ($s_0=0$ if there are no such followers). At stage $s_0$, all followers for $P^e$ are cancelled. At the next \stage after stage $s_0$, $P^e$ will appoint a new follower. This follower can never be cancelled. This shows that $H^e$ is nonempty. 	
	
	Let $s_1$ be the last stage at which any follower in $H^e$ receives attention. We claim that $P^e$ does not appoint followers after stage $s_1$. For if it does, let $s$ be the least stage greater than $s_1$ at which $P^e$ appoints a follower $x$. At stage $s_1$, all followers weaker than the followers in $H^e$ are cancelled, and so $x$ is stronger than any other follower for $P^e$ (at any stage $t\ge s$) other than the followers in $H^e$. As the followers in $H^e$ do not require attention after stage $s_1$, $x$ can never be cancelled. But this means that $x\in H^e$, for a contradiction with the maximality of $s_1$. 
	
	The fact that $P^e$ does not appoint any followers after stage $s_1$ implies that no followers for $P^e$ require attention after stage $s_1$; so $P^e$'s overall action is finitary. It also implies that the weakest follower $x\in H^e$ is never realised. For if it is, then all followers in $H^e$ are sometime realised, and $P^e$ would be instructed to appoint a new follower. As $x\notin E$, this shows of course that $E\ne \varphi^e$, and so $P^e$ is met. 	
\end{proof}

It follows that $E$ is not computable. 

\subsubsection*{Reductions}

\begin{lemma}\label{lem_MP_non_permission_non_enumeration}
	Let $s$ be a \emph{stage}, and let $x\in F_s$. Let $i\in R_s(x)$, and suppose that $v^i_s(x)\converge$. Suppose that $A^i\rest{v^i_s(x)} = A^i_s\rest{v^i_s(x)}$. Then $x\notin E$. 
\end{lemma}

\begin{proof}
	By induction on $r> s$ we can see that if $x\in F_s$ then $i\in R_r(x)$ and $k^i_r(x) = k^i_s(x)$, and so that $t^i_r(x) = t^i_s(x)$; that $v^i_r(x)\converge = v^i_s(x)$, and that if $i= \top_{r-1}(x)$ then $x$ is not permitted at stage $r$. 
	
	The point is that if $x\in E$ then there must be some stage $r\ge s$ such that $i = \top_r(x)$; if $i\ne \top_s(x)$, then $i= \top_r(x)$ where $r$ is the next stage at which $x$ receives attention. But then, the fact that $x$ will not be permitted once $i$ becomes $\top(x)$, means that $x$ cannot be enumerated into $E$. 
	
	This, of course, is where we use the fact that $A^i$ is c.e., rather than merely $\Delta^0_2$. 
\end{proof}

The next lemma states that unless cancelled, enumerated into $E$, or given open permission from $A^1$, the use $v^i(x)$ of reducing the statement $x\notin E$ to $A^i$ stabilizes. Using the notation above, let $H$ be the collection of followers which are never cancelled nor enumerated into $E$. For $x\in H$, let $R_\w(x) = \lim_s R_s(x)$ be the collection of $i<2$ such that $i\in R_s(x)$ for all but finitely many $s$. For $i\in R_\w(x)$, let $k^i_\w(x) = \lim_s k^i_s(x)$, and so on. 

\begin{lemma}\label{lem_MP_final_uses}
	Let $x\in H$ and let $i\in R_\w(x)$. There is some \stage $s$ such that $x\in F_s$, $v^i_s(x)\converge$ and $A^i\rest{v^i_s(x)} = A^i_s\rest{v^i_s(x)}$. 
\end{lemma}

\begin{proof}
	Let $k = k^i_\w(x)$ and let $t = t^i_\w(x) = t^i(k,x)$. There are two cases. 
	
	\medskip
	
	If either $x\in G^i_\w(k)$, or $x\in L^i_\w(k)$ and is never realised, then Lemma \ref{lem_MP_no_changes_to_psi} implies that for all $z\in M^i(k,x)$, $\psi^i(z) = t$. As $T^i$ traces $\psi^i$, we have $t\in T^i(z)$ for all $z\in M^i(k,x)$. There is a \stage $s>t$ at which for all $z\in M^i(k,x)$, $t\in T^i_s(z)$ by an $A^i$-correct computation. Then $s$ is as required. 
	
	\medskip
	
	In the second case, $x\in KL^i_w(k)$ and $x$ is eventually realised. Let $s>t$ be a \stage such that $x\in \bar{KL^i_s}(k)$. By Lemma \ref{lem_MP_followers_in_KLG_are_traced}, $v^i_s(x)\converge$. We claim that $s$ is a stage as required by the lemma. For if not, there is a least \stage $r>s$ by which we see an $A^i$-change below $v^i_s(x)$. Since by induction $i = \top_{r-1}(x)$, at stage $r$, $x$ would be permitted, contradicting the definition of $k$. 	
\end{proof}

\begin{lemma}\label{lem_MP_first_reduction}
	$E\le_\Tur A^0$. 
\end{lemma}

The point is that for $x\in H$ we always have $0\in R_\w(x)$. 

\begin{proof}
	Let $x<\w$. To decide, with oracle $A^0$, whether $x\in E$ or not, first see if $x\in F_{x}$. If not, then $x\in E$ if and only if $x\in E_{x}$.
	
	Suppose that $x\in F_x$. By Lemmas \ref{lem_MP_followers_move_finitely_much} and \ref{lem_MP_final_uses}, with oracle $A^0$, we can find a \stage $s>x$ at which one of the following hold:
	\begin{itemize}
		\item $x\notin F_s$ (that is, $x$ has been cancelled by stage $s$).
		\item $x\in E_s$ (that is, $x$ has been enumerated into $E$ by stage $s$).
		\item $v^0_s(x)\converge$, and $A^0\rest{v^0_s(x)} = A^0_s\rest{v^0_s(x)}$. 
	\end{itemize} 
	By Lemma \ref{lem_MP_non_permission_non_enumeration}, this allows $A^0$ to decide whether $x\in E$ or not. 
\end{proof}	
	
The difference between $A^0$ and $A^1$ is that there may be $x\in H$ such that $1\notin R_\w(x)$: these are the followers that receive open permission from $A^1$ but do not get later permission from $A^0$. 

\begin{lemma}\label{lem_MP_main_A1_lemma}
	There are only finitely many followers $x\in H$ such that $1\notin R_\w(x)$. 
\end{lemma}

\begin{proof}
	Suppose that $x\in H$ and $1\notin R_\w(x)$. Then $x\in K^0_\w(c)$. By Lemma \ref{lem_MP_bound_on_K}, $|K^0_\w(c)| \le c$. 
\end{proof}

\begin{lemma}\label{lem_MP_second_reduction}
	$E\le_\Tur A^1$.
\end{lemma}

\begin{proof}
	Similar to the proof of Lemma \ref{lem_MP_first_reduction}. For a follower $x\in F_x$, we search for a \stage $s>x$ by which it is either cancelled, enumerated, or $v^1_s(x)\converge$ and $A^1_s$ is correct up to this value. Lemma \ref{lem_MP_main_A1_lemma} says that this search will terminate for all but finitely many followers $x$, and so non-uniformly will give a method for reducing $E$ to $A^1$. 
\end{proof}

\section{Proof of Theorem \ref{thm_main}}

In this section we adapt the construction of the previous section and provide a construction of a noncomputable c.e.\ set $E$, computable from every SJT-hard c.e.\ set, thus proving Theorem \ref{thm_main}. 

\subsection{Discussion}

There is only one really new ingredient, and our treatment is not too surprising to those familiar with $\Pi^0_2$ constructions on trees. Instead of being given two (or finitely many) SJT-hard c.e.\ sets, together with traces for $\Sigma^0_2$ functions we approximate, we need to guess, among all pairs of c.e.\ sets and possible traces, which indeed trace the functions that we enumerate. The construction is performed on a tree of strategies. Nodes $\tau$ on the tree will test if there are infinitely many $\tau$-\emph{stages}, at which we can calculate uses of reducing $x\in E$ to the corresponding c.e.\ set $W^e$. 

The small degree of non-uniformity which was necessary in the construction of the previous section plays an important role. A follower $x$ for some node $\s$ on the tree can only be cleared by only finitely many nodes $\tau$ for which $\s$ guesses that there are infinitely many $\tau$-\emph{stages}. In other words, it requires eventual permission from only finitely many c.e.\ sets $W^e$. Other SJT-hard sets do not comprehend $x$'s existence. This is akin to those sets giving $x$ immediate open permission. The tree machinery ensures that each such set is troubled by at most finitely many such followers.

\subsection{Construction}

As before, we enumerate a set $E$. To ensure that $E$ is noncomputable, we meet the same positive requirements $P^e$ as in the previous section, which state that $E\ne \vphi^e$, where $\seq{\vphi^e}_{e<\w}$ is a list of all partial computable functions. 

\

Let $\seq{W^e}_{e<\w}$ be an effective list of all c.e.\ sets. Shortly we will define, for all $e<\w$, an order function $h^e$. Let, uniformly in $e$, $\seq{T^{e,c}}_{c<\w}$ list all $W^e$-c.e.\ traces which are bounded by $h^e$. During the construction we define, uniformly in $e$, a partial $\Sigma^0_2$ function $\psi^e$. The negative requirements are named $N^{e,c}$, and state that if $T^{e,c}= \seq{T^{e,c}(z)}_{z<\w}$ traces $\psi^e$, then $E$ is computable from $W^e$. 

\

The construction takes place on a tree of strategies. The definition of the tree is recursive: given a node (a strategy) on the tree, the immediate successors of the node on the tree are determined by the possible outcomes of the node. If a node $\s$ works for $P^e$, then it has a single outcome. If a node $\tau$ works for $N^{e,c}$, then $\tau$ has two outcomes, $\infty$ and $\finn$. The outcome $\infty$ is stronger than the outcome $\finn$, and this ordering induces a total priority ordering on the tree. The outcome $\infty$ indicates that there are infinitely many $\tau$-\emph{stages}. 

Let $\PPP^e$ be the collection of nodes on the tree that work for $P^e$, and $\NNN^{e,c}$ be the collection of nodes that work for $N^{e,c}$. We let and $\NNN^e = \bigcup_c \NNN^{e,c}$, $\NNN = \bigcup_e \NNN^e$, and $\PPP=\bigcup_e \PPP^e$. 

To complete the definition of the tree of strategies, we need to show how to assign requirements to nodes. We could simply assign each level of the tree a single requirement. However, for simplicity of presentation, we would like to assume that for all $\s\in \PPP$ there is some $\tau\in \NNN$ such that $\conc{\tau}{\infty}\subseteq \s$. The easiest way to achieve this is by recursively assign requirements to nodes during the definition of the tree; to each node $\rho$ we assign the strongest requirement (from an $\w$-list of all requirements) which has not yet been assigned to any proper initial segment of $\rho$, subject to the restriction that if there is no $\tau$, which has been already placed in $\NNN$, such that $\conc{\tau}{\infty}\subseteq \rho$, then we must assign a negative requirement to $\rho$. After verifying that there is a true path, we will easily see that the true path contains a node of the form $\conc{\tau}{\infty}$ for some $\tau\in \NNN$ (as there are SJT-hard c.e.\ sets), and this would allow us to show that every requirement is assigned to some node on the true path. 

\subsubsection*{The order functions} 

The next order of business is defining the order functions $h^e$. These derive from the structure of the tree and the intended structure of the boxes. The idea is that a follower $x$ for a node $\s\in \PPP$ needs to be cleared by all $\tau\in \NNN$ such that $\conc{\tau}{\infty}\subseteq \s$: if $\tau\in \NNN^e$ then $W^e$-permission is required. The guess by $\s$ that there are infinitely many $\tau$-\emph{stages} allows for the machinery of the previous section to operate smoothly. 

The search over all traces $T^{e,c}$ (for $c<\w$) means that the various nodes $\tau\in \NNN^e$ have to cooperate in defining a single function $\psi^e$: each $\tau$ gets its own column to play with. A node $\tau$ will require a number of $k$-boxes for various $k$; as there are infinitely many $\tau$'s in $\NNN^e$, to keep $h^e$ well-defined, the smallest $k$ such that $\tau$ requests $k$-boxes needs to increase with $\tau$. For convenience of notation, we let each $\tau$ request $k$-boxes for $k\ge |\tau|$. 

For a fixed $\tau$ and $k\ge |\tau|$, how many $k$-boxes? We need to count the number of possible followers that progress down the chain of boxes, similarly to what has been done in the justification of the previous section. In other words, we need to calculate bounds on the sizes of sets $G^\tau_s(k)$, $K^\tau_s(k)$ and $L^\tau_s(k)$ which are the analogues of the sets $G^i_s(k)$, $K^i_s(k)$ and $L^i_s(k)$ of the previous section. We will still have $|K^\tau_s(k)| \le k$ as this is bounded by the potential size of the trace. However, it is no longer true that every $x\in G^\tau_s(k)$ is an element of the same $K^\rho_s(k')$ (for $k'\in \{k,k+1\}$): more than two sets mean that $\top_s(x)$ may have value among a number of nodes $\rho$ of length at most $k'\le k+1$. The number of these strings is bounded by the number of nodes of length $k+1$; as the tree is at most binary branching, the number of such nodes is bounded by $2^{k+1}$. 

We also need to bound $L^\rho_s(k)$ for such strings $\rho$; to avoid too meticulous an examination of the way requirements are assigned to nodes, we allow more than one node $\s$ require a private $k$-box from $\rho$. In general, $\s\in \PPP$ will require a private $|\s|$-box from the longest $\tau\in \NNN$ such that $\conc{\tau}{\infty}\subseteq \s$. So for $\rho\in \NNN$ and $k>|\rho|$, let $\Theta^\rho(k)$ be the collection of nodes $\s\in \PPP$ of length $k$ for which $\rho$ is the longest string in $\NNN$ such that $\s$ extends  $\conc{\rho}{\infty}$; of course $\left|\Theta^\rho(k)\right|\le 2^k$. Counting all these contributions, we let, for all $k<\w$,
\[ a(k) =  1+ 2^{k+2}(k+2)(1+2^{k+1}),\]
which will be a bound on $G^\tau_s(k)$. Then any $\tau\in \NNN$ will require $|\Theta^\tau(k)| + (a(k))^{k+1}\le 2^k + (a(k))^{k+1}$ many $k$-boxes for $k\ge |\tau|$. 

We can now define the order functions $h^e$. Partition $\w$ into $\w$ many columns $\w^{[\tau]}$, indexed by $\tau\in \NNN^e$. For each $\tau\in \NNN^e$, partition $\w^{[\tau]}$ into intervals $I^\tau(k)$ for $k\ge |\tau|$, such that for all $k\ge |\tau|$, 
\[ | I^\tau(k) | = 2^k + a(k)^{k+1} .\] 
Let $h^e$ be an order function such that for all $\tau\in \NNN^e$, for all $k\ge |\tau|$, for all $x\in I^\tau(k)$, $h^e(x)\le k$. The fact that for all $k$, there are only finitely many nodes $\tau\in \NNN^e$ of length at most $k$, implies that such an order function can be found, and in fact defined effectively given $e$.

\subsubsection*{Local traces and \emph{stages}}

The nodes $\tau\in \NNN^e$ collaborate in defining the function $\psi^e$. We let $\psi^\tau = \psi^e\rest{\w^{[\tau]}}$; the node $\tau$ is responsible for defining $\psi^\tau$. The function $\psi^\tau$ is defined to be the partial limit of a uniformly computable sequence $\seq{\psi^\tau_s}_{s<\w}$. 

For $\tau\in \NNN$, we will shortly define the collection of $\tau$-\emph{stages}. To define $\psi^\tau_s$, we start with $\psi^\tau_0(z)=0$ for all $z\in \w^{[\tau]}$. If $s>0$ is not a $\tau$-\emph{stage}, then we let $\psi^{\tau}_{s} = \psi^\tau_{s-1}$. Thus if there is a last $\tau$-\emph{stage} $s$, then $\psi^\tau = \psi^\tau_s$. If $s$ is a $\tau$-\emph{stage}, then we may redefine $\psi^\tau_s(z)=s$ for some $z\in \w^{[\tau]}$; for all other $z$, we let $\psi^{\tau}_s(z) = \psi^\tau_{s-1}(z)$. 

Let $\tau\in \NNN^{e,c}$. For brevity, we let $T^\tau = T^{e,c}\rest{\w^{[\tau]}}$. Thus if $T^{e,c}$ is a trace for $\psi^e$, then $T^\tau$ is a trace for $\psi^\tau$. For $z\in \w^{[\tau]}$ and $s<\w$, we let $T^\tau_s(z)$ be the collection of numbers enumerated into $T^\tau(z)$ by stage $s$ with oracle $W^e_s$. We may assume that for all $z$ and $s$, $|T^\tau_s(z)|\le h^e(z)$, so for all $k\ge |\tau|$, for all $z\in I^\tau(k)$, $|T^\tau_s(z)|\le k$. For $t\in T^\tau_s(z)$, we let $u^\tau_s(z,t)$ be the $W^e_s$-use of enumerating $t$ into $T^\tau_s(z)$. 

The collection of $\tau$-\emph{stages} depends on whether $\tau$ is \emph{accessible} at stage $s$, a notion which we define later. Given this, we define the collection of $\tau$-\emph{stages}. For all $\tau$, 0 is a $\tau$-\emph{stage}. Let $s>0$, and let $\tau\in \NNN$. If $\tau$ is not accessible at stage $s$, then $s$ is not a $\tau$-\emph{stage}. Suppose that $\tau$ is accessible at stage $s$; let $r$ be the previous $\tau$-\emph{stage}. Then $s$ is a $\tau$-\emph{stage} if for every $z\in \w^{[\tau]}$ mentioned by stage $r$, $\psi^\tau_r(z) = \psi^\tau_{s-1}(z)$ is an element of $T^\tau_s(z)$.

\subsubsection*{Followers}

Nodes $\s\in \PPP$ appoint followers. For any follower $x$, we let $\s(x)$ be the node which appointed $x$. A follower $x$ for $\s\in \PPP^e$ is \emph{realised} at stage $s$ if $\vphi^e_s(x)=0$. The requirement $P^e$ is \emph{satisfied} at stage $s$ if there is some $x\in E_s$ such that $\vphi^e_s(x)=0$. If the requirement $P^e$ is satisfied at stage $s$, then no node $\s\in \PPP^e$ takes any action at stage $s$. 

With any follower $x$, alive at the end of a stage $s$, we associate auxiliary objects.
\begin{itemize}
	\item We attach a nonempty set of nodes $R_s(x)\subset \NNN$. For all $\tau\in R_s(x)$, $\s(x)$ extends $\conc{\tau}{\infty}$. The set $R_s(x)$ is the set of nodes which need to clear $x$ before it is enumerated into $E$. 
	\item We define a node $\top_s(x)\in R_s(x)$. This is the node from which $x$ next requires permission. 
	\item For all $\tau\in R_s(x)$, we define a number $k^\tau_s(x)\ge |\tau|$. This is the level at which $x$ points. 
	\item For all $\tau\in R_s(x)$, we define a box $M^\tau_s(x)\subset I^\tau(k^\tau_s(x))$. 
	\item For all $\tau\in R_s(x)$, we define a value $t^\tau_s(x)<\w$. 
\end{itemize}

Suppose that a follower $x$ is alive at the end of stage $s$, and let $\tau\in R_s(x)$. Suppose that $s$ is a $\tau$-\emph{stage}. We say that a $\tau$-computation is defined for $x$ at stage $s$ if for all $z\in M^\tau_s(x)$ we have $t^\tau_s(x)\in T^\tau_s(z)$. In this case, we let 
\[ v^\tau_s(x) = \max \left\{ u^\tau_s\left(z,t^\tau_s(x)\right)  \,:\, z\in M^\tau_s(x) \right\}.\] We denote the fact that a $\tau$-computation is defined for $x$ at stage $s$ by writing $v^\tau_s(x)\converge$; otherwise, we write $v^\tau_s(x)\diverge$. 

Let $s$ be a stage at the beginning of which a follower $x$ is alive, and suppose that $s$ is a $\top_{s-1}(x)$-\emph{stage}; let $\tau = \top_{s-1}(x)$.  Let $r$ be the previous $\tau$-\emph{stage} before $s$. We say that $x$ is \emph{permitted} at stage $s$ if $v^i_r(x)\diverge$ or if $W^e_s\rest{v^\tau_{r}(x)} \ne W^e_r\rest{v^\tau_r(x)}$, where $\tau\in \NNN^e$. Note again that a follower $x$ is permitted at stage $s$ only if $s$ is a $\top_{s-1}(x)$-\emph{stage}; but $x$ may be permitted at a stage $s$ at which $\s(x)$ is not accessible: recall that $\s(x)$ properly extends the node $\top_{s-1}(x)$. 

All followers alive at stage $s$ are linearly ordered by priority, which is given to followers appointed earlier; again, at most one new follower is appointed at each stage. New followers are chosen large, and so the priority ordering coincides with the natural ordering on natural numbers. 

The general structure of the stage is as follows. During stage $s$ we inductively define the collection of nodes which are accessible at that stage. Once a node $\s\in \PPP$ has been declared accessible, we see if it wants to appoint a new follower or not; if it does, it ends the stage, and cancels followers for all weaker nodes. Otherwise, it lets its only child be accessible. Once a node $\tau\in \NNN$ has been declared accessible, we decide if $s$ is a $\tau$-\emph{stage} or not. If not, and $|\tau|<s$, then we let $\conc{\tau}{\finn}$ be next accessible. If $|\tau|=s$ then we end the stage. If $s$ is a $\tau$-\emph{stage}, then we see if $\tau$ permits any realised follower $x$. If so, then the strongest such follower $x$ receives attention and is moved; all followers weaker than $x$ are cancelled, and the stage is ended. If no follower is permitted, then $\conc{\tau}{\infty}$ is next accessible. 

We explain why it is important that if $s$ is a $\tau$-\emph{stage}, and $\tau$ permits a realised follower $x$ at stage $s$, then we let $x$ move, even if $\s(x)$ is not accessible at stage $s$. For $\tau$ permitting $x$ means that $t^\tau(x)$ is extracted from $T^\tau(z)$ for some $z\in M^\tau_s(x)$. The set $W^e$ (where $\tau\in \NNN^e$) needs to decide \emph{now} whether to assign a new use for reducing $x\in E$ to $W^e$ using the same boxes, or promote $x$. If $x$ is not allowed to move, then this incident may repeat indefinitely, meaning that the use goes to infinity. After all, $W^e$ does not know if $\s(x)$ will ever be accessible again. 

Of course, if $x$ is permitted but is unrealised, then as it lies in its private box $\{z^\tau(\s(x))\}$, no weaker follower changes the value of $\psi^\tau(z)$, and so in this case $\tau$ can appoint a new use based on the same value of $\psi^\tau(z)$, which will stabilise the use. Once a follower is in $K^\tau(k)$ and not in $L^\tau(k)$ (it is moved from private to public boxes), a weaker follower $y$ such that $M^\tau(k,y)\subset M^\tau(k,x)$ may force a change in the values of $\psi^\tau(z)$ for some $z\in M^\tau(k,x)$, which lands $\tau$ in the position of having to promote $x$ if it is permitted.

\subsubsection*{Carving the boxes}

For all $\tau\in \NNN$ and $k\ge |\tau|$, let $J^\tau(k)$ be a subset of $I^\tau(k)$ of size $2^k$, and let $B^\tau(k,\seq{}) = I^\tau(k)\setminus J^\tau(k)$; so $|B^\tau(k,\seq{})| = a(k)^{k+1}$. As in the previous section, we recursively define $B^\tau(k,\alpha)$ for $\alpha\in {}^{\le k+1}a(k)$ by starting with $B^\tau(k,\seq{})$ and splitting $B^\tau(k,\alpha)$, which inductively has size $a(k)^{k+1-|\alpha|}$, into $a(k)$-many subsets $B^\tau(k,\conc{\alpha}{m})$ (for $m<a(k)$) of equal size $a(k)^{k-|\alpha|}$. Again, for $\alpha,\beta\in {}^{\le k+1}a(k)$, if $\alpha\subseteq \beta$ then $B^\tau(k,\alpha)\supseteq B^\tau(k,\beta)$, and if $\alpha\perp \beta$ then $B^\tau(k,\alpha)\cap B^\tau(k,\beta) = \emptyset$. For all $\alpha\in {}^{\le k+1}a(k)$, $B^\tau(k,\alpha)\cap J^\tau(k)= \emptyset$. 

Recall that $\Theta^\tau(k)$ is a set (possibly empty) of nodes $\s\in \PPP$ of length $k$; so $|\Theta^\tau(k)|\le 2^k = |J^\tau(k)|$. We fix an injection $\s\mapsto z^\tau(\s)$ of $\Theta^\tau(k)$ into $J^\tau(k)$. 

Let $\tau\in \NNN$. For any follower $x$, alive at the end of stage $s$, such that $\tau\in R_s(x)$, either 
\begin{itemize}
	\item $M^\tau_s(x) = \left\{ z^\tau(\s(x))  \right\}$ ($x$ resides in $\s$'s private box at stage $s$); in this case $\tau = \top_s(x)$ and $s$ has not moved since it was appointed; or
	\item $M^\tau_s(x) = B^\tau(k,\alpha)$ for some string $\alpha\in {}^{\le k+1}a(k)$; we denote this string by $\alpha^\tau_s(x)$. 
\end{itemize}

Suppose that $s$ is a $\tau$-\emph{stage}. Once we know which followers are alive at the end of the stage, for $k\ge |\tau|$ we then define a string $\beta^\tau_s(k)$:
\begin{itemize}
	\item If there is a follower $y$, alive at the beginning and the end of stage $s$, such that $\top_{s-1}(y)= \tau$, $k^\tau_{s-1}(y) = k$, and $M^\tau_{s-1}(y) \nsubset J^\tau(k)$ (in the notation of the justifications, $y\in \bar{K^\tau_s}(k)$), then we let $\beta^\tau_s(k) = \alpha^\tau_{s-1}(y)$ for the weakest such follower $y$. 
	\item If there is no such follower $y$, then we let $\beta^\tau_s(k) = \seq{}$. 
\end{itemize}

The main bookkeeping lemma is:

\begin{lemma}\label{lem_M_bookkeeping}
	If $\tau\in \NNN$ and $s$ is a $\tau$-\emph{stage}, then for all $k\ge |\tau|$, $|\beta^\tau_s(k)|\le k$, and there is some $m<a(k)$ such that for all followers $y$ which are alive ay both the beginning and the end of stage $s$, if $\tau\in R_s(y)$, $k^\tau_{s-1}(y) = k$, and $M^\tau_{s-1}(y)\nsubset J^\tau(k)$, then $\alpha^\tau_{s-1}(y)\ne \conc{\beta^\tau_s(k)}{m}$. We let $m^\tau_s(k)$ be the least such $m$. 
\end{lemma}

\subsubsection*{Construction}

At stage $s>0$ we recursively define the finite path of nodes which are accessible at stage $s$. The root $\seq{}$ is always accessible.

\medskip

First, suppose that a node $\tau\in \NNN$ is accessible at stage $s$. We check to see if $s$ is a $\tau$-\emph{stage} or not, as described above. If not, and $|\tau|<s$, then we let $\conc{\tau}{\finn}$ be accessible next; if $|\tau|=s$ then we end the stage, and cancel all followers for nodes that lie to the right of $\tau$. If $s$ is a $\tau$-\emph{stage}, but there is no realised follower which $\tau$ permits at stage $s$, then we let $\conc{\tau}{\infty}$ be accessible next (unless $|\tau|=s$, when we end the stage and cancel followers for nodes that lie to the right of $\tau$). Suppose, then, that $\tau$ permits realised followers at stage $s$. Let $x$ be the strongest realised follower which is permitted by $\tau$ at stage $s$. We cancel all followers weaker than $x$. We then promote $x$ as follows:

\medskip

\noindent {\textbf{1.}} If $R_{s-1}(x) = \{\tau\}$, this means that only $\tau$ cares about $x$; since $\tau$ has just permitted $x$, we enumerate $x$ into $E$. If $\s(x)\in \PPP^e$ then the requirement $P^e$ is now satisfied, so we cancel all followers for $\s(x)$. 
	
\medskip

\noindent {\textbf{2.}} If $\tau$ is not the only element of $R_{s-1}(x)$, but $k^\tau_{s-1}(x) = |\tau|$, then $\tau$ can promote $x$ no more, and so gives it open permission: we let $R_s(x) = R_{s-1}(x)\setminus\{\tau\}$. The parameters $k^\rho_s(x)$, $M^\rho_s(x)$ and $t^\rho_s(x)$ remain unchanged for all $\rho\in R_s(x)$. We choose a new value for $\top(x)$: we let $\top_s(x)$ be the longest $\rho\in R_s(x)$ for which $k^\rho_s(x)$ is maximal among the elements of $R_s(x)$. 
	
\medskip

\noindent {\textbf{3.}}	
 Otherwise, we let $\tau$ promote $x$. We let $R_s(x) = R_{s-1}(x)$, and $k^\tau_s(x) = k^\tau_{s-1}(x)-1$. We let $M^\tau_s(x) = B^\tau(k, \conc{\beta^\tau_s(k)}{m^\tau_s(k)})$ for $k = k^\tau_s(x)$. We let $t^\tau_s(x) = s$. We define $\psi^\tau_s(z) = s$ for all $z\in M^\tau_s(x)$. The parameters $k^\rho_s(x)$, $M^\rho_s(x)$ and $t^\rho_s(x)$ remain unchanged for other nodes $\rho\in R_s(x)$. We do, however, pick a new value for $\top(x)$ as we did in the second case: we let $\top_s(x)$ be the longest $\rho\in R_s(x)$ for which $k^\rho_s(x)$ is maximal among the elements of $R_s(x)$. 

\medskip

We then end the stage. 

\medskip

Suppose now that $\s\in \PPP$ is accessible at stage $s$. If there is some follower for $\s$ which is still unrealised at stage $s$, then $\s$ takes no action, and lets its only child be accessible (unless $|\s|=s$, in which case we end the stage). Otherwise, $\s$ appoints a new follower $x$, of large value. We cancel all followers for all nodes weaker than $\s$. We then set up $x$'s parameters as follows:

\begin{itemize}
	\item We let $R_s(x)$ be the collection of all nodes $\tau\in \NNN$ such that $\conc{\tau}{\infty}\subseteq \s$. By the way we distributed the requirements on the tree, we see that $R_s(x)$ is nonempty. 
	\item We let $\top_s(x)$ be the longest node in $R_s(x)$. 
	\item For all $\tau\in R_s(x)$, we let $k^\tau_s(x) = |\s|$. Note, of course, that as $\tau\subset \s$ we get $k^\tau_s(x)\ge |\tau|$. 
	\item We let $t^\tau_s(x) = s$ for all $\tau\in R_s(x)$. 
	\item For $\tau = \top_s(x)$, we let $M^\tau_s(x) = \{ z^\tau(\s)\}$, the singleton subset of $J^\tau(|\s|)$ which is reserved for $\s$. 
	\item For $\tau\in R_s(x)\setminus \{\top_s(x)\}$ we let $M^\tau_s(x) = B^\tau(|\s|,\conc{\beta^\tau_s(|\s|)}{m^\tau_s(|\s|)})$. 
\end{itemize}

We then end the stage. 

\subsection{Justification}

We need to prove Lemma \ref{lem_M_bookkeeping} to show that the construction can be performed as prescribed. Much of the argument mimics the argument given in the previous section, and so we give the definitions and notation, and then only highlight the new ingredients. We start though by tracking the possible combinations for the function $\tau\mapsto k^\tau_s(x)$ on $R_s(x)$.

For $\s\in \PPP$, we let $F^\s_s$ be the collection of followers for $\s$ which are alive at the end of stage $s$. We let $F_s = \bigcup_{\s\in \PPP}F^\s_s$. 

We note that if $r<s$ and $x\in F_r\cap F_s$, then $R_s(x)\subseteq R_r(x)$. If $\tau\in R_s(x)$ and $k^\tau_s(x) = k^\tau_r(x)$, then $M^\tau_s(x) = M^\tau_r(x)$ and $t^\tau_s(x) = t^\tau_r(x)$. Hence for $k= k^\tau_s(x)$ we let $M^\tau(k,x) = M^\tau_s(x)$ and $t^\tau(k,x) = t^\tau_s(x)$. If $M^\tau(k,x)\ne \{z^\tau(\s(x))\}$, then we let $\alpha^\tau(k,x) = \alpha^\tau_s(x)$.

\begin{lemma}\label{lem:k and k-top}
	Let $x\in F_s$, and let $k = k^{\top_s(x)}_s(x)$. Let $\tau\in R_s(x)$. 
	\begin{enumerate}
		\item If $\tau\subseteq \top_s(x)$, then $k^\tau_s(x)  = k$. 
		\item If $\top_s(x)\subsetneq \tau$, then $k^\tau_s(x) = k-1$. 
	\end{enumerate}
	Also, if $\tau\in R_{s-1}(x)\setminus R_s(x)$, then $\tau$ is the longest string in $R_{s-1}(x)$. 
\end{lemma}

\begin{proof}
	The second part of the lemma follows from the first. Suppose that $\tau\in R_{s-1}(x)\setminus R_s(x)$; so $\tau = \top_{s-1}(x)$ and $x$ is promoted at stage $s$. We have $k = k^\tau_{s-1}(x)= |\tau|$. Suppose that $\rho\in R_{s-1}(x)$. If $\tau\subsetneq \rho$, then by the first part of the lemma, applied at stage $s-1$, we have $k^\rho_{s-1}(x) = k-1$, which is smaller than $|\rho|$, which is impossible. 
	
	Now we prove the first part of the lemma, by induction on $s$. 
	
	If $x$ is appointed at stage $s$, then for all $\tau\in R_s(x)$, $k^\tau_s(x) = k = |\s(x)|$, and $\top_s(x)$ is the longest element of $R_s(x)$. 
	
	Suppose that $x$ is promoted at stage $s$. Let $\bar \tau = \top_{s-1}(x)$ and let $\bar k = k^{\bar \tau}_{s-1}(x)$. For all $\tau\in R_s(x)$ different from $\bar \tau$ we have $k^\tau_s(x) = k^\tau_{s-1}(x)$. 
	
	If $\bar\tau\notin R_s(x)$, then by the second part of the lemma, $\bar \tau$ is the longest node in $R_{s-1}(x)$. Hence for all $\tau\in R_s(x)$ we have $k^\tau_s(x) = \bar k$, and $\top_s(x)$ is chosen to be the longest string in $R_s(x)$. 
	
	Suppose then that $\bar \tau\in R_s(x)$, so $R_s(x)= R_{s-1}(x)$. We set $k^{\bar \tau}_s(x) = \bar k-1$. There are two cases. 
	\begin{itemize}
		\item If $\bar \tau$ is the shortest string in $R_s(x)$, then for all $\tau\in R_s(x)$ we have $k^{\tau}_s(x) = \bar k-1$; we then choose $\top_s(x)$ to be the longest node in $R_s(x)$. 
		\item Otherwise, for $\tau\in R_s(x)$, if $\tau\supseteq \bar \tau$ then $k^\tau_s(x) = \bar k -1$, and if $\tau\subsetneq \bar \tau$ then $k^\tau_s(x) = \bar k$. We choose $\top_s(x)$ to be the immeidate predecessor of $\bar \tau$ in $R_s(x)$. 
	\end{itemize}
\end{proof}

In particular, Lemma \ref{lem:k and k-top} shows that if $x$ is promoted at stage $s$, then $\top_s(x)\ne \top_{s-1}(x)$. 

\

Let $\tau\in \NNN$, $k\ge |\tau|$ and $s<\w$. 
\begin{itemize}
	\item We let $K^\tau_s(k)$ be the collection of followers $x\in F_s$ such that $\tau = \top_s(x)$, $k = k^\tau_s(x)$, and $M^\tau(k,x)\ne \{z^\tau(\s(x))\}$. 
	\item We let $L^\tau_s(k)$ be the collection of followers $x\in F_s$ such that $\tau = \top_s(x)$, $k = k^\tau_s(x)$, and $M^\tau(k,x) = \{z^\tau(\s(x))\}$.
	\item We let $G^\tau_s(k)$ be the collection of followers $x\in F_s$ such that $\tau\in R_s(x)$, $k = k^\tau_s(x)$, but $\tau \ne \top_s(x)$. 	
\end{itemize}

Again we use the notation $KG^\tau_s(k)$ to denote the union $K^\tau_s(k)\cup G^\tau_s(k)$, and the notation $\bar{K^\tau_s}(k) = K^\tau_{s-1}(k)\cap K^\tau_s(k)$ etc. Again, $KLG^\tau_s(k)$ is the collection of followers $x\in F_s$ such that $\tau\in R_s(x)$ and $k= k^\tau_s(x)$, $KL^\tau_s(k)$ is the collection of followers $x\in KLG^\tau_s(k)$ such that $\tau = \top_s(x)$, and $KG^\tau_s(k)$ is the collection of followers $x\in KLG^\tau_s(k)$ such that $\alpha^\tau(k,x)$ is defined. 

The following lemma translates the construction into this terminology:

\begin{lemma}\label{lem_MA_construction_lemma}
	Let $x\in KLG^\tau_s(k)$. Let $t = t^\tau(k,x)$. The stage $t$ is the least stage at which $x\in KLG^\tau(k)$. At stage $t$, $x$ is placed into $LG^\tau(k)$. 
	\begin{itemize}
		\item If $x$ was appointed at stage $t$, then $x$ is placed into $L^\tau_t(k)$ if $\tau = \top_t(x)$, and into $G^\tau_t(k)$ if not. 
		\item Otherwise, $x$ is realised at stage $t$, and $x$ is added to $G^\tau_t(k)$. 
	\end{itemize}
Unless $x\in L^\tau_t(k)$, at stage $t$ we let $M^\tau(k,x) = B^\tau(k, \conc{\beta^\tau_t(k)}{m^\tau_t(k)})$. The string $\beta^\tau_t(k)$ is defined as follows:
\begin{itemize}
	\item If $\bar{K^\tau_t}(k) = \emptyset$, then $\beta^\tau_t(k) = \seq{}$. 
	\item If $\bar{K^\tau_t}(k) \ne \emptyset$, then $\beta^\tau_t(k) = \alpha^\tau(k, \max \bar{K^\tau_t}(k))$. 
\end{itemize}
The number $m^\tau_t(k)$ is the least $m<a(k)$ such that for all $y\in \bar{KG^\tau_t}(k)$, $\alpha^\tau(k,y)\ne \conc{\beta^\tau_t(k)}{m^\tau_t(k)}$. 
\end{lemma}
 
\

The arguments of the justification for the minimal pair construction of the previous section now carry through, where $i<2$ is replaced by $\tau\in \NNN$ and \emph{stage} is replaced by $\tau$-\emph{stage}. This gives us analogues of all lemmas from Lemma \ref{lem_MP_marker_size_obeys_priority} to Lemma \ref{lem_MP_bound_on_K}, including the first part of  Lemma \ref{lem_M_bookkeeping}, that $|\beta^\tau_s(k)|\le k$.

\begin{lemma}\label{lem_MA_bound_on_L}
	For all $\s\in \Theta^\tau(k)$, $\left| F^\s_s\cap L^\tau_s(k)\right| \le k+1$.
\end{lemma}

\begin{proof}
	Identical to the proof of Lemma \ref{lem_MP_bound_on_L}.
\end{proof}

Let $b(k) = (k+1)(1+2^k)$. So $a(k) = 1+2^{k+2}\cdot b(k+1)$. Since $\left|\Theta^\tau(k)\right|\le 2^k$, Lemma \ref{lem_MA_bound_on_L} and the analogue of Lemma \ref{lem_MP_bound_on_K} tell us that $\left| KL^\tau_s(k)\right| \le b(k)$. 

\begin{lemma}\label{lem_MA_bound_on_G}
	$\left| G^\tau_s(k) \right| < a(k)$.
\end{lemma}

\begin{proof}
	Let $x\in G^\tau_s(k)$; let $\rho = \top_s(x)$. By Lemma \ref{lem:k and k-top}, $x\in KL^\rho_s(k)$ or $x\in KL^\rho_s(k+1)$. This implies that $|\rho|\le k+1$, so there are at most $2^{k+2}$ many possibilities for such strings $\rho$. We have seen that $|KL^\rho_s(k)|$ and $|KL^\rho_s(k+1)|$ are bounded by $b(k)$ and $b(k+1)$ respectively, and so both bounded by $b(k+1)$. Hence $|G^\tau_s(k)|\le 2^{k+2}b(k+1) = a(k)-1$. 
\end{proof}

The argument of the previous section now gives a proof of Lemma \ref{lem_M_bookkeeping}.

\subsection{Verification}

\begin{lemma}\label{lem_MA_followers_move_finitely_much}
	Every follower receives attention only finitely many times. 
\end{lemma}

\begin{proof}
	Identical to the proof of Lemma \ref{lem_MP_followers_move_finitely_much}, using $k_s(x) = \sum_{\tau\in R_s(x)}k^\tau_s(x)$. 
\end{proof}

The analogue of Lemma \ref{lem_MP_infinitely_many_stages} is the fact that the true path is infinite. Recall that a node $\rho$ lies on the true path if $\rho$ is accessible at infinitely many stages, but there are only finitely many stages at which some node that lies to the lexicographic left of $\rho$ is accessible. The true path is a linearly ordered initial segment of the tree of strategies. 

\begin{lemma}\label{lem_true_path_N}
	If $\tau\in \NNN$ lies on the true path, then one of $\tau$'s children lies on the true path as well. 
\end{lemma}

\begin{proof}
	Suppose, for contradiction, that the lemma fails. This means that there is some stage $s_0$ such that for every $s\ge s_0$, if $\tau$ is accessible at stage $s$ then $s$ is a $\tau$-\emph{stage} and $\tau$ permits some follower at stage $s$. 
	
	Let $KLG^\tau_s = \bigcup_{k\ge |\tau|} KLG^\tau_s(k)$. If $x\in KLG^\tau_s \setminus KLG^\tau_{s-1}$ then some $\s\supseteq \conc{\tau}{\infty}$ is accessible at stage $s$. This means that $s<s_0$. That is, no new followers are added to $KLG^\tau_s$ after stage $s_0$. By Lemma \ref{lem_MA_followers_move_finitely_much}, each follower in $KLG^\tau_{s_0}$ receives attention finitely many times. We reach a contradiction. 	
\end{proof}

Again, for $X\in \{K,L,G\}$, let $X^\tau_\w(k) = \lim_s X^\tau_s(k)$. For each $\tau\in \NNN$, $k\ge |\tau|$ and $X\in \{K,L,G\}$, we see that $X^\tau_\w(k)$ is finite. 

\begin{lemma}\label{lem_true_path_P}
	Suppose that $\s\in \PPP^e$ lies on the true path. Then the requirement $P^e$ is met, and there is a stage after which no follower for $\PPP^e$ ever requires attention. The unique child of $\s$ also lies on the true path. 
\end{lemma}

\begin{proof}
	This is proved by induction on $|\s|$, using the argument of Lemma \ref{lem_MP_reqs_are_met_and_fair}.
\end{proof}

As a corollary, we see that the true path is infinite.

\begin{lemma}\label{lem_SJT_hard_on_true_path}
	Let $\tau\in \NNN^{e,c}$ be on the true path. Suppose that $T^{e,c}$ is a trace for $\psi^e$. Then $\conc{\tau}{\infty}$ lies on the true path. 
\end{lemma}

\begin{proof}
	By Lemma \ref{lem_true_path_N}, it suffices to show that there are infinitely many $\tau$-\emph{stages}. This follows from the fact that $T^\tau$ traces $\psi^\tau$. 
\end{proof}

\begin{lemma}\label{lem_MA_noncomputable}
	$E$ is not computable. 
\end{lemma}

\begin{proof}
	We need to show that every requirement $P^e$ is met. By Lemma \ref{lem_true_path_P}, it suffices to show that for all $e$ there is some node $\s\in \PPP^e$ on the true path. This follows from the fact that the true path is infinite, and from the way we distributed requirements to nodes, once we see that there is some node $\tau\in \NNN$ such that $\conc{\tau}{\infty}$ lies on the true path. This follows from Lemma \ref{lem_SJT_hard_on_true_path}, and the fact that SJT-hard c.e.\ sets exist. 
\end{proof}

\subsubsection*{Reductions}

We turn to show that $E$ is computable from every SJT-hard c.e.\ set. Suppose that $W^e$ is SJT-hard; so there is some $c<\w$ such that $T^{e,c}$ traces $\psi^e$. As the true path is infinite, find some $\tau\in \NNN^{e,c}$ on the true path. By Lemma \ref{lem_SJT_hard_on_true_path}, $\conc{\tau}{\infty}$ lies on the true path; there are infinitely many $\tau$-\emph{stages}.

\begin{lemma}\label{lem_MA_top_will_come}
	Let $x\in F_s$ and suppose that $\tau\in R_s(x)$. If $x\in E$, then there is some stage $r\ge s$ such that $\tau = \top_r(x)$. 
\end{lemma}

\begin{proof}
	Let $v>s$ be the stage at which $x$ is enumerated into $E$. We have $R_{v-1}(x) = \{\top_{v-1}(x)\}$. If $\tau = \top_{v-1}(x)$ we are done. Otherwise, since $\tau\notin R_{v-1}(x)$, let $r\ge s$ be the last stage at which $\tau\in R_r(x)$. Then $\tau = \top_{r}(x)$. 
\end{proof}

\begin{lemma}\label{lem_MA_non_permission_non_enumeration}
	Let $s$ be a $\tau$-\emph{stage}. Let $x\in F_s$ such that $\tau\in R_s(x)$ and suppose that $v^\tau_s(x)\converge$. Suppose that $W^e\rest{v^\tau_s(x)} = W^e_s\rest{v^\tau_s(x)}$. Then $x\notin E$. 
\end{lemma}

\begin{proof}
	Identical to the proof of Lemma	\ref{lem_MP_non_permission_non_enumeration}, using Lemma \ref{lem_MA_top_will_come}.
\end{proof}

Let $H$ be the collection of followers which are never cancelled nor enumerated into $E$. We use notation, such as $R_\w(x)$, similar to the notation we used before. 

\begin{lemma}\label{lem_MA_final_uses}
	Let $x\in H$ and let $\tau\in R_\w(x)$. There is some $\tau$-\emph{stage} $s$ such that $x\in F_s$, $v^\tau_s(x)\converge$ and $W^e\rest{v^\tau_s(x)} = W^e_s\rest{v^\tau_s(x)}$.
\end{lemma}

\begin{proof}
	Identical to the proof of Lemma \ref{lem_MP_final_uses}, using the assumption that $T^\tau$ traces $\psi^\tau$. 
\end{proof}

\begin{lemma}\label{lem_MA_reduction}
	$E\le_\Tur W^e$. 
\end{lemma}

\begin{proof}
	Similar to the proofs of Lemmas \ref{lem_MP_first_reduction} and \ref{lem_MP_second_reduction}. 
	
	Let $x\in F_x$. To find, with oracle $W^e$, whether $x\in E$ or not, wait for a $\tau$-\emph{stage} $s>x$ at which one of the following holds:
	\begin{itemize}
		\item $x$ has been cancelled by stage $s$;
		\item $x\in E_s$; or
		\item $\tau\in R_s(x)$, $v^\tau_s(x)\converge$ and $W^e\rest{v^\tau_s(x)} = W^e_s\rest{v^\tau_s(x)}$. 
	\end{itemize}
	If such a stage $s$ is found, then by Lemma \ref{lem_MP_non_permission_non_enumeration}, $W^e$ can decide at stage $s$ if $x\in E$ or not. 
	
	\medskip
	
	We claim that such a stage $s$ can be found for all but finitely many followers $x$. First, note that there are only finitely many followers $x\in \bigcup_s F_s$ such that $\s(x)$ is stronger than $\conc{\tau}{\infty}$: those $\s$ that lie to the left of $\tau$ are visited only finitely many times, and those that are extended by $\tau$ lie on the true path, and so appoint only finitely many followers by Lemma \ref{lem_true_path_P}. 
	
	Hence, we let $x\in \bigcup_s F_s$ and assume that $\s(x)$ is not stronger than $\conc{\tau}{\infty}$. If $\s(x)$ lies to the right of $\conc{\tau}{\infty}$, then any $\conc{\tau}{\infty}$ stage $s>x$ satisfies the condition above: either $x$ is cancelled by stage $s$, or $x\in E_s$. Suppose, then, that $\s(x)\supseteq \conc{\tau}{\infty}$. Thus, $\tau\in R_t(x)$ where $t$ is the stage at which $x$ is appointed. 
	
	If $x\notin H$ then a stage as above is definitely found. Suppose that $x\in H$. If $\tau$ is never removed from $R(x)$, then Lemma \ref{lem_MA_final_uses} ensures that a stage $s$ as above is found. Otherwise, let $\rho = \top_\w(x)$. Let $r+1$ be the stage at which $\tau$ is removed from $R(x)$; so $k^\tau_r(x) = |\tau|$. Lemma \ref{lem:k and k-top} shows that $\rho\subset \tau$, and that $k^\rho_r(x) = |\tau|$. So $k^\rho_\w(x) \le |\tau|$. Thus
	\[ x\in \bigcup_{\rho\subset \tau, \rho\in \NNN} \bigcup_{k\le |\tau|} KL^\rho_\w(k).\]
However, this set is the finite union of finite sets, and so is finite. 	
\end{proof}

% 
% 
% \bibliographystyle{plain}
% \bibliography{../../biblist}

\end{document}